\documentclass[12pt, a4paper, oneside, final ,reqno]{amsart}

\usepackage[T1]{fontenc}
\usepackage[utf8]{inputenc}
\usepackage[dvipsnames]{xcolor}
\usepackage[english]{babel}
\usepackage{mathtools, amsthm, amssymb, amscd} 
\usepackage{newtxtext, newtxmath} 
\usepackage[DIV=11]{typearea} 
\usepackage{hyperref} 
\usepackage[notref,notcite]{showkeys} 
\usepackage{todonotes}
\usepackage{enumitem} 
\usepackage{xypic} 
\usepackage{tikz}
\usetikzlibrary{cd, decorations.markings, arrows} 
\usepackage{tikz-cd}
\usepackage[capitalize]{cleveref}

\DeclareMathAlphabet{\mathcal}{OMS}{cmsy}{m}{n} 

\definecolor{DarkPurple}{rgb}{0.40,0.0,0.20}
\hypersetup{
	colorlinks =true,
	linkcolor= DarkPurple, 
	citecolor = NavyBlue 
} 

\usetikzlibrary{cd, decorations.markings, arrows, matrix, calc}

\newcommand{\Aut}{\operatorname{Aut}}

\newcommand{\N}{\mathbb{N}}

\newcommand{\A}{\mathcal{A}}

\newcommand{\X}{\mathcal{X}}
\newcommand{\Cred}{C_{r}^{*}}

\newcommand{\Lip}{\operatorname{Lip}}

\newcommand{\CC}{{\mathbb{C}}}

\newcommand{\vertiii}[1]{{\left\vert\kern-0.25ex\left\vert\kern-0.25ex\left\vert #1 
		\right\vert\kern-0.25ex\right\vert\kern-0.25ex\right\vert}}
\newcommand{\Bvert}[1]{{\Big\vert\kern-0.25ex\Big\vert\kern-0.25ex\Big\vert #1 
		\Big\vert\kern-0.25ex\Big\vert\kern-0.25ex\Big\vert}}
\newcommand{\bvert}[1]{{\big\vert\kern-0.25ex\big\vert\kern-0.25ex\big\vert #1 
		\big\vert\kern-0.25ex\big\vert\kern-0.25ex\big\vert}}
\newcommand{\nvert}[1]{{\vert\kern-0.25ex\vert\kern-0.25ex\vert #1 
		\vert\kern-0.25ex\vert\kern-0.25ex\vert}}

\renewcommand{\leq}{\leqslant}
\renewcommand{\geq}{\geqslant}

\newcommand{\cd}{\cdot}

\newcommand{\ot}{\otimes}
\newcommand{\otm}{\otimes_{\mathrm{min}}}
\newcommand{\hot}{\widehat \otimes}

\newcommand{\op}{\oplus}

\newcommand{\ci}{\circ}

\newcommand{\ti}{\times}
\newcommand{\nn}{\mathbb{N}}
\newcommand{\zz}{\mathbb{Z}}

\newcommand{\al}{\alpha}
\newcommand{\be}{\beta}
\newcommand{\ga}{\gamma}
\newcommand{\Ga}{\Gamma}
\newcommand{\de}{\delta}

\newcommand{\ep}{\varepsilon}

\newcommand{\io}{\iota}

\newcommand{\la}{\lambda}

\newcommand{\Na}{\nabla}

\newcommand{\si}{\sigma}

\newcommand{\te}{\theta}

\newcommand{\ze}{\zeta}
\newcommand{\pa}{\partial}

\newcommand{\C}[1]{\mathcal{#1}}
\newcommand{\T}[1]{\text{#1}}

\newcommand{\B}[1]{\mathbb{#1}}

\newcommand{\R}{\mathrm}

\newcommand{\fork}[2]{\left\{ \begin{array}{#1} #2 \end{array} \right.}

\newcommand{\ma}[2]{\left(\begin{array}{#1} #2 \end{array} \right)}

\newcommand{\lrar}{\Leftrightarrow}

\newcommand{\su}{\subseteq}

\newcommand{\wit}{\widetilde}

\newcommand{\inn}[1]{\langle #1 \rangle}

\newcommand{\binn}[1]{\big\langle #1 \big\rangle}

\newcommand{\sem}{\setminus}

\newcommand{\ralg}{\rtimes_\alg}

\newcommand{\alg}{{\operatorname{alg}}}

\newtheorem{lemma}{Lemma}[section]

\newtheorem{theorem}[lemma]{Theorem}
\newtheorem{proposition}[lemma]{Proposition}

\theoremstyle{definition}
\newtheorem{definition}[lemma]{Definition}
\newtheorem{example}[lemma]{Example}
\newtheorem{remark}[lemma]{Remark}

\usepackage{thmtools}
\declaretheorem[style=theorem,name={Theorem}]{theoremletter}

\newtheorem{introcorollary}[theoremletter]{Corollary}

\title[]{Quantum metrics on crossed products with groups of polynomial growth}

\date{\today}

\author[1]{Are Austad}
\address{Are Austad, Department of Mathematics, University of Oslo, P.O. Box 1053 Blindern, N-0316 Oslo, Norway}
\email{areaus@math.uio.no}

\author[2]{Jens Kaad}
\address{David Kyed, Department of Mathematics and Computer Science, University of Southern Denmark, Campusvej 55, DK-5230 Odense M, Denmark}
\email{kaad@imada.sdu.dk}

\author[1]{David Kyed}
\address{David Kyed, Department of Mathematics and Computer Science, University of Southern Denmark, Campusvej 55, DK-5230 Odense M, Denmark}
\email{dkyed@imada.sdu.dk}

\keywords{Crossed products, compact quantum metric spaces, spectral triples }
\subjclass[2020]{Primary: 58B34; Secondary: 47L65, 46L55} 

\numberwithin{equation}{section} 

\begin{document}
	
	\maketitle
	\begin{abstract}
We show how to equip the crossed product between a group of polynomial growth and a compact quantum metric space with {a compact quantum metric space structure}. When the quantum metric on the base space arises from a spectral triple, {which is compatible with the action of the group}, we furthermore show that the crossed product becomes a spectral metric space. 
Lastly, we analyse {the spectral triple} at the crossed product level from the point of view of unbounded $KK$-theory and show that it {arises} as an internal Kasparov product of unbounded Kasparov modules. 
\end{abstract}
\tableofcontents	
	\section{Introduction}
%
Rieffel's theory of compact quantum metric spaces provides an elegant framework for studying non-commutative geometry from a metric point view. There is by now an extensive body of work on quantum metric spaces, and the theory has provided a mathematical foundation for rigorously proving heuristic approximation results from physics -- the seminal result in this direction  that ``matrix algebras converge to the sphere'' was proven by Rieffel in \cite{Rie:MSG}. The early years of compact quantum metric spaces have been dominated by {a} case-by-case analysis of interesting examples and only a limited amount of tools are available for constructing new compact quantum metric spaces out of old ones.  
However, research into such general constructions  is currently gaining momentum, and in recent years there has been an increased focus on the non-commutative (metric) geometry of crossed products \cite{CMRV2008, HSWZ:STC,BMR:DSS, KK:DCQ, Klisse2023}.
The present paper continues this line of research, by showing how recent approximation techniques from \cite{KaadKyed2022} and \cite{Kaad2023} can be employed to treat crossed products with finitely generated groups of polynomial growth, thus providing the first results about the compact quantum metric structure of crossed products beyond the realm of (virtually) abelian groups. \\

The theory of compact quantum metric spaces is strongly related to Connes' non-commutative geometry, and one of the goals of the present paper is to {emphasise} this {relationship} in the particular case of crossed product algebras.  In vague terms, the problem that one encounters in the non-commutative differential geometric setting is the following: how to fuse a spectral triple on a group $C^*$-algebra with a spectral triple on a {$C^*$-algebra} acted upon by the group, to obtain a spectral triple on the resulting crossed product.   
This question has been investigated in several papers \cite{HSWZ:STC,BMR:DSS,CMRV2008, Pat:CST}, and to obtain a satisfying answer, one needs to enforce additional assumptions on the action.

More formally, the starting point of the investigation is a {countable} discrete group $\Gamma$ equipped with a proper length function  $l\colon \Gamma \to \B R$, from which one obtains a spectral triple  $(\B C\Gamma, \ell^2(\Gamma), D_l)$ {on the reduced group $C^*$-algebra} $\Cred(\Gamma)$. The Dirac operator $D_l$ is the closure of the essentially selfadjoint operator $\mathcal{D}_l\colon C_c(\Gamma) \to \ell^2(\Gamma)$ defined as $\mathcal{D}_l(\delta_s):=l(s)\delta_s$ for all $s\in \Gamma$. 
This construction dates back to the seminal paper by Connes \cite{Con:CFH}, but to {increase the flexibility we are in this paper working in the more general context} of matrix-valued length functions $l\colon \Gamma \to M_n( \B C)$ (see \cref{d:length}), so that the Dirac operator acts on {$\ell^2(\Gamma) \ot \B C^{\oplus n}$. Indeed, without this added flexibility of matrix-valued length functions, it would be more difficult to tackle basic examples such as the standard spectral triple on the noncommutative $2$-torus, as illustrated in \cref{ss:length}. 
	Given an action $\al$ of $\Ga$ on a $C^*$-algebra $A$ we show how to construct an unbounded Kasparov module $( A \ralg \Ga, \ell^2(\Ga,A) \ot \B C^n, D_l)$ which connects the reduced crossed product $A \rtimes \Ga$ to the base $C^*$-algebra $A$.}

{Consider now a unital spectral triple $(\C A, H, D)$ {on the $C^*$-algebra $A$}, which is still endowed with the action $\al$ of $\Ga$. If} the action is \emph{smooth}, i.e.~preserves the {unital $*$-subalgebra $\A \su A$}, then the {algebraic crossed product} $\C A \rtimes_{\textrm{alg}}\Gamma$ {forms} a natural dense $*$-subalgebra in $A\rtimes\Gamma$.  Hence, the natural problem that arises at this point is the following:
\begin{center}
\emph{Can we combine {the Dirac operators} $D_l$ and $D$ into a single Dirac operator for the crossed product?}
 \end{center}
The natural candidate is the so-called tensor sum $D_l\times_{\Na} D$, which acts on a Hilbert space $G$, which is either $\ell^2(\Gamma, H^{\oplus n})$ or $\ell^2(\Gamma, H^{\oplus n})^{\oplus 2}$ depending on the {parities} of the spectral triples involved; see \cref{ss:triples-on-crossed} for the relevant definitions. {An investigation of commutators between the tensor sum and elements in the algebraic crossed product does however reveal that an extra condition on the action of $\Gamma$ has to be satisfied. More precisely, letting $d \colon \C A \to \B L(H)$ denote the derivation coming from the spectral triple $(\C A,H,D)$ it has to be assumed that} the action is \emph{metrically equicontinuous}, meaning that $\sup_{g\in \Gamma}\big\| d(\alpha_g(a)) \big\| <\infty$. Under this extra assumption, it holds that $(\C A\rtimes_{\alg} \Gamma, G, D_l\times_{\Na} D )$ is a {unital} spectral triple {on the reduced crossed product}; see \cite[Theorem 2.7]{HSWZ:STC}. {In fact, in \cref{sec:unbounded-KK} we shall see that this spectral triple can be viewed as the unbounded Kasparov product of the unbounded Kasparov module $( \C A \ralg \Ga, \ell^2(\Ga,A^{\op n}), D_l)$ and the unital spectral triple $(\C A,H,D)$; see also the end of the introduction for some further preliminary discussion of this result.} \\

Turning instead to the non-commutative \emph{metric} approach, one passes from the spectral triple $(\C A,H,D)$ to the pair $(\A, L_D)$ where $L_{D}$ denotes the seminorm on $\C A$ defined as {$L_D(a):=\|d(a)\|$}. Rieffel's fundamental insight is, that in order to capture a non-commutative analogue of a metric on a compact space, one needs to require that the \emph{Monge-Kantorovi\v{c} metric}
\[
\rho_D(\mu,\nu):= \sup\{|\mu(a)-\nu(a)| \mid a\in \A, \ L_D(a)\leq 1\}
\]
metrises the weak$^*$ topology on the state space $S(A)$. In this case, $(\C A, H, D)$  is referred to as a \emph{spectral metric space}.

These considerations {can be made substantially more general} by replacing $(\C A,L_D)$ with a pair consisting of an operator system  $\X$ and a seminorm $L_\X$ whose associated Monge-Kantorovi\v{c} metric metrises the weak$^*$ topology on the state space $S(\X)$. {In this way, one} arrives at Rieffel's notion of a \emph{compact quantum metric space} \cite{Rie:MSS, Rie:MSA, Rie:GHD, Rie:CQM} {(formulated for operator systems).} The theory of compact quantum metric spaces has been extensively developed by Rieffel himself, with important contributions from Li \cite{Li:DCQ, Li:ECQ, Li:GH-dist} and an impressive body  of work centred around Latr\'{e}moli\`ere \cite{Lat:AQQ, Lat:QGH, Lat:DGH, Lat:GPS, Lat:MGP, LatPack:Solenoids, LatAgu:AF, Lat:Compactness-theorem} whose approach is based on $C^*$-algebras rather than operator systems. To align with current trends in non-commutative geometry  \cite{walter:GH-convergence, walter-connes:truncations, CvS:Truncations, GvS:TolFin, walter-tori, Rie:Fourier-truncations}, we have {opted for the general setting} of operator systems in the present text, and our operator systems will always be concrete, in the sense that they come equipped with an inclusion into a fixed $C^*$-algebra.

So if $\X \subseteq A$ is an operator system, one may again consider an action of $\Gamma$ on $A$ which preserves the given operator system $\X$. In this situation, we naturally obtain a new operator system $\X \rtimes_\alg\Gamma$ sitting inside the reduced crossed product $A \rtimes\Gamma$. If $(\X,L_{\C X})$ is {a compact quantum metric space} and if {$l \colon \Ga \to M_n(\B C)$ is a proper matrix length function such that $(\B C \Ga, L_{D_l})$ is a compact quantum metric space,} the natural question is:
\begin{center}
\emph{{Can we} fuse the seminorm $L_{D_l}$ and the seminorm $L_\X$ to obtain a compact quantum metric space structure on the crossed product operator system $\X \rtimes_\alg\Gamma$?}
\end{center}
In the special situation where $\Gamma$ is the group of integers, positive results {were} obtained in \cite{KK:DCQ} in the context of abstract compact quantum metric spaces, and in \cite{HSWZ:STC} and \cite{BMR:DSS} in the context of spectral metric spaces.  The results from \cite{HSWZ:STC} and \cite{BMR:DSS} were recently generalised to  virtually abelian groups in \cite{Klisse2023}, but beyond this not much seems to be known about the quantum metric structure of crossed products.  Moreover, the techniques applied in the papers \cite{KK:DCQ, BMR:DSS, HSWZ:STC, Klisse2023} do not immediately extend to treat more general groups; see e.g.~\cite[Section 2]{HSWZ:STC} for a discussion of these limitations. \\

Our first theorem generalises \cite[Theorem A]{KK:DCQ}, by providing a {positive answer to our question} in the situation where $\Gamma$ is just assumed to be amenable. {In order to state our first theorem properly, we spend some time introducing two interesting seminorms on the crossed product operator system $\C X \ralg \Ga$.

The first seminorm $L_l$ on $\C X \ralg \Ga$ comes from the unbounded Kasparov module $( A \ralg \Ga, \ell^2(\Ga,A) \ot \B C^n, D_l)$, meaning that $L_l(x)$ agrees with the operator norm of the closure of the commutator $[D_l,x]$.

The second seminorm $L_H$ needs a bit more preparation. Recall that a norm $\nvert{-}$ on the compactly supported functions $C_c(\Gamma, \CC)$ is called \emph{order-preserving} if $\nvert{f} \leq \nvert{g}$ whenever $0\leq f(s) \leq g(s)$ for all $s\in \Gamma$. Natural examples  of such norms include  all $p$-norms for $p\in {[1,\infty]}$. Applying such an order-preserving norm, we may lift a seminorm $L_{\C X} \colon \C X \to [0,\infty)$ to a seminorm $L_H$ on the operator system $\C X \ralg \Ga$. Indeed, identifying $\C X \ralg \Ga$ with the compactly supported maps $C_c(\Ga,\C X)$ we have a well-defined composition $L_H(x) := \nvert{L_{\C X}\ci x}$.}


 	\begin{theoremletter}\label{thm:main-theorem-crossed}
	Let $\Ga$ be an amenable, countable discrete group endowed with a proper matrix-valued length function $l\colon \Gamma \to M_n(\CC)$ {and let $\nvert{\cdot }$ be an order-preserving norm on $C_c(\Gamma, \CC)$.}  {Let moreover $\C X \su A$ be an operator system which is preserved by an action of $\Ga$ on the unital $C^*$-algebra $A$. Suppose now that $(\X,L_\X)$ and $(\CC \Ga, L_{D_l})$ are compact quantum metric spaces. Then the seminorm} $L\colon \X\rtimes_{\alg} \Gamma \to [0,\infty)$ defined as
		\begin{equation}\label{eq:intro-defi-of-L}
			L(x) := \max\big\{{L_l(x), \ L_H(x), \  L_H(x^*)}  \big\},
		\end{equation}
	provides $\X\rtimes_\alg \Gamma$ with {the structure of a compact quantum metric space}. 
	\end{theoremletter}

We sketch the key ideas behind the proof of this theorem. The amenability assumption on $\Ga$ is crucial as it guarantees the existence of a Følner sequence $(F_n)_{n=1}^\infty$, from which we may define the associated Berezin transforms $(\beta_n)_{n=1}^\infty$, see \cref{def:Berezin-transform-on-crossed-product}. Applying the Berezin transforms, we approximate the pair $(\X\rtimes_{\alg} \Gamma, L)$ by systems $(\X\rtimes_{\alg} S_n, L)$, where $S_n := F_n F_n^{-1} \su \Ga$ is a finite subset determined by the $n^{\T{th}}$ entry of the Følner sequence. The approximation is done through the technical result \cref{prop:CQMS-if-close-to-CQMS}, which morally says that if a system is arbitrarily close to a compact quantum metric space, it must itself be a compact quantum metric space. Pairs such as $(\X\rtimes_{\alg} S_n, L)$ are significantly easier proven to be compact quantum metric spaces since $S_n$ is finite, see \cref{prop:Berezin-image-is-CQMS}. Indeed, this follows pretty much directly from $(\X,L_\X)$ being a compact quantum metric space. Much of the proof now requires careful analysis of the Berezin transforms. In particular, we must guarantee that it is a contraction with respect to the seminorm $L$ from \eqref{eq:intro-defi-of-L}, as well as make sure that the norm $\Vert z - \beta_n (z)\Vert$ for $z \in \X\rtimes_{\alg} \Gamma$  is small when $n$ is sufficiently large. The latter again turns out to  hinge on amenability of $\Ga$ and the assumption that $(\CC \Ga, L_{D_l})$ is a compact quantum metric space, see \cref{prop:Berezin-norm-Lip-approx}. Note that even if $\X$ is an algebra, our approximation techniques depend crucially on the use of operator systems, as $\X \rtimes_{\alg} S_n$ is in general not multiplicatively closed for our sets $S_n$.

The main source of examples satisfying the assumptions of  \cref{thm:main-theorem-crossed} comes from groups of polynomial growth, whose word length functions were shown to yield compact quantum metric spaces in \cite[Theorem 1.4]{ChristRieffel2017}.  We single this case out as a separate corollary for  ease of reference.

\begin{introcorollary}\label{maincor:poly-growth-cor-cqms}
Let $\Gamma$ be a finitely generated group of polynomial growth and {let} $l\colon \Gamma \to [0,\infty)$ be the word length function arising from a finite generating set. {Let moreover $\C X \su A$ be an operator system which is preserved by an action of $\Ga$ on the unital $C^*$-algebra $A$. If $(\X,L_\X)$ is a compact quantum metric space, then} the seminorm $L$ defined in \eqref{eq:intro-defi-of-L} yields a compact quantum metric space structure on $\X\rtimes_\alg \Gamma$.
\end{introcorollary}

{Our second main result is in fact a corollary to \cref{thm:main-theorem-crossed} but we state it here as a separate theorem in order to highlight its importance. This second main result is a generalisation of the existing results in \cite{Klisse2023,HSWZ:STC, BMR:DSS}} 
which cover the case of $n=1$ and $\Gamma$ virtually abelian.


\begin{theoremletter}\label{introthm:spectral-metric-space}
Let $(\C A, H, D)$ be a {spectral metric space structure on a $C^*$-algebra $A$} and let $\Gamma$ be an amenable, countable discrete group equipped with a proper matrix-valued length function $l\colon \Gamma \to M_n(\CC)$ for which $(\CC\Ga, {\ell^2(\Ga) \ot \B C^{n}}, D_l)$ is a spectral metric space {structure on the reduced group $C^*$-algebra}. If $\Gamma$ acts on $A$ in a smooth and metrically equicontinuous manner, then the resulting spectral triple $(\A\rtimes_\alg\Gamma, G, D_l\times_{\Na} D)$ yields a spectral metric space structure on the reduced crossed product $A \rtimes \Ga$. 
\end{theoremletter}	

The proof of \cref{introthm:spectral-metric-space} boils down to bounding the seminorm arising from the spectral triple $(\A\rtimes_\alg\Gamma, G, D_l\times_{\Na} D)$ from below by a seminorm of the form \eqref{eq:intro-defi-of-L}. The latter will induce a compact quantum metric space by \cref{thm:main-theorem-crossed}, from which \cref{introthm:spectral-metric-space} follows immediately. 

In particular, the above \cref{introthm:spectral-metric-space} applies in the case where $\Gamma$ is a group of polynomial growth and $l$ is a word length function arising from a finite generating set. Note that by Gromov's celebrated theorem \cite{gromov-groups-of-polynomial-growth-and-expanding-maps}, being of polynomial growth is equivalent to being virtually nilpotent, and hence this  may be viewed as  a natural extension of the main result from \cite{Klisse2023} from virtually abelian to virtually nilpotent groups. 
Moreover, the statements about groups of polynomial growth
hold true in the more general context where $\ell\colon \Gamma \to [0,\infty)$ is a proper length function of \emph{bounded doubling}, since these are the groups treated in \cite[Theorem 1.4]{ChristRieffel2017}
and amenability is implied by the assumption of bounded doubling. 
For finitely generated groups, bounded doubling and polynomial growth coincide, as shown in \cite[Proposition 1.2]{ChristRieffel2017}.\\

  {Let us end this introduction by some further remarks on the unital spectral triple $(\C A \ralg \Ga, G, D_l \ti_{\Na} D)$ on the reduced crossed product $A \rtimes \Ga$. Recall that this spectral triple makes sense under the assumption that $(\C A,H,D)$ is a unital spectral triple on a $C^*$-algebra $A$ which is equipped with a smooth and metrically equicontinuous action of $\Ga$. Moreover, we require the existence of a proper matrix length function $l \colon \Ga \to M_n(\B C)$. As demonstrated in \cref{introthm:spectral-metric-space}, this spectral triple on the reduced crossed product becomes a spectral metric space under some natural conditions on the defining data. Another point of the present paper is to further motivate the particular choice of Dirac operator $D_l \ti_{\Na} D$ by showing that it arises naturally via an application of the internal Kasparov product in unbounded $KK$-theory, \cite{Kuc:KKU,Mes:UKC,KaLe:SFU,MeRe:NMU,MeLe:SRS}.

    In \cref{ss:length}, we show how $( A \rtimes_{\alg} \Gamma, \ell^2(\Gamma, A) \ot \B C^n, D_l)$  provides an unbounded Kasparov module from $A\rtimes \Gamma$ to $A$ and, by definition, the spectral triple $(\A, H, D)$ provides one from $A$ to $\B C$. Recall from \cite{BaJu:TBK} that unbounded Kasparov modules give rise to classes in $KK$-theory via the Baaj-Julg bounded transform and we therefore obtain a class in $KK$-theory $[D_l] \in KK^p(A \rtimes \Ga, A)$ and a class in $K$-homology $[D] \in K^q(A)$, where the parities $p$ and $q$ in $\{0,1\}$ depend on the parities of the involved unbounded Kasparov modules. The main goal of \cref{sec:unbounded-KK} is to verify the compatibility of these structures, by proving Theorem \ref{t:unbprod} here below.} This is carried out by applying the machinery of unbounded $KK$-theory as introduced in \cite{Mes:UKC} and further developed in \cite{KaLe:SFU,MeRe:NMU}. In particular, we construct a correspondence, in the sense of \cite[Definition 6.3]{KaLe:SFU}, from the unbounded Kasparov module to the spectral triple. The main ingredient in this correspondence is a Hermitian connection $\Na$ which can be applied to lift the abstract Dirac operator $D$ on $H$ to a selfadjoint unbounded operator on the internal tensor product $(\ell^2(\Gamma,A) \ot \B C^n)\hot_A H$. The domain of this Hermitian connection is a certain completion of the algebraic direct sum $C_c(\Gamma,\C A) \ot \B C^n \su \ell^2(\Gamma,A) \ot \B C^n$ which takes into account the differentiable structure of the norm-dense unital $*$-subalgebra $\C A \su A$ as witnessed by the selfadjoint unbounded operator $D$. The conditions on the Hermitian connection ensures that the corresponding tensor sum (or unbounded product operator) $D_l \ti_\Na D$ fits as the abstract Dirac operator of a spectral triple on the algebraic crossed product $\C A \rtimes_\alg \Gamma$. The fact that this spectral triple represents the internal Kasparov product is then really a consequence of a criterion established by Kucerovsky in \cite{Kuc:KKU}.
 

\begin{theoremletter}\label{t:unbprod}
Suppose that the action of $\Ga$ on $A$ is {smooth and metrically} equicontinuous with respect to a unital spectral triple $(\C A,H,D)$ of parity $q$. Suppose moreover that $l \colon \Ga \to M_n(\B C)$ is a proper matrix length function of parity $p$. Then it holds that $(\C A \ralg \Ga, G, D_l {\ti_\Na} D)$ is a unital spectral triple of parity $p + q$ (modulo $2$). Moreover, in the case where $A$ is separable, the corresponding $K$-homology class $[D_l \ti_\Na D] \in K^{p + q}(A \rtimes \Ga)$ agrees with the internal Kasparov product $[D_l] \hot_A [D]$.
\end{theoremletter}
	
The rest {of the} paper is structured as follows: in \cref{sec:prelim}, we recall the basics pertaining to crossed products and compact quantum metric spaces, and in \cref{sec:qm-on-crossed} we introduce our key tool, a version of the Berezin transform, and prove the central technical result of the paper,  \cref{thm:mainA}. In \cref{sec:slip-from-length}, we analyse the seminorm on the crossed product arising from a length function on a group, and prove \cref{thm:main-theorem-crossed} and \cref{maincor:poly-growth-cor-cqms}. We then turn to  the analysis of spectral metrics on crossed products in \cref{ss:specross} and prove \cref{introthm:spectral-metric-space}, while \cref{sec:unbounded-KK} is devoted to the $KK$-theoretic point of view on the Dirac operator on the crossed product and a proof of \cref{t:unbprod}. This structure is chosen so that readers that are not familiar with the techniques of unbounded $KK$-theory may still access the results regarding compact quantum metric structures.	
	
	\subsubsection*{Conventions}
	All inner products on Hilbert spaces and Hilbert $C^*$-modules will be assumed linear in the second argument. For a Hilbert $C^*$-module $E$, we let $\B L(E)$ denote the unital $C^*$-algebra of bounded adjointable operators on $E$. 
	Regarding tensor products, we will be using the symbol $\otimes$ to denote algebraic tensor products, the symbol $\hat{\otimes}$ to denote tensor products between Hilbert spaces and Hilbert $C^*$-modules, while the symbol $\otimes_{\R{min}}$ refers to the minimal tensor product between $C^*$-algebras. 
        {The unique $C^*$-norm on a $C^*$-algebra $A$ is always denoted by $\| \cd \|_\infty$ and this notation therefore also applies to the operator norm.}
	{For each $n \in \N$ we consider $\B C^n$ as a Hilbert space and let $e_j \in \B C^n$ denote the $j^{\T{th}}$ standard basis vector for $j \in \{1,\ldots,n\}$.}
         We consistently apply the same notation ``$1$'' for the identity operator on a vector space and the unit in a unital $C^*$-algebra.

	\subsubsection*{Acknowledgements} The authors gratefully acknowledge the financial support from  the Independent Research Fund Denmark through grant no.~9040-00107B and 1026-00371B, and from the the ERC through the MSCA Staff Exchanges grant no.~101086394.

	\section{Preliminaries}\label{sec:prelim}
	
\subsection{Reduced crossed products}\label{ss:cross}
Let $\Ga$ be a countable discrete group and let $\alpha\colon \Gamma \to \Aut(A)$ be an action of $\Gamma$ on a unital $C^*$-algebra by {$*$-automorphisms}. 
 In this preliminary subsection we (very briefly) recall the definition of the reduced crossed product $A \rtimes \Ga$ and the corresponding dual coaction, taking the theory of Hilbert $C^*$-modules as our point of departure. {The reader who is less familiar with Hilbert $C^*$-module could use the book \cite{Lance1995} as a reference.}

We let $\ell^2(\Ga,A)$ denote the Hilbert $C^*$-module over $A$ consisting of sequences $(a_s)_{s\in \Gamma}$   
indexed by $\Ga$ with values in $A$, subject to the convergence constraint that the series $\sum_{s \in \Ga} a_s^* a_s^{\phantom{\ast}}$ converges {with respect to the norm} on $A$. For each $s\in \Gamma$, we denote by $\delta_s\in  \ell^2(\Ga,A)$  the {sequence} which is zero everywhere but in position $s$, where its value is the unit in $A$. The \emph{reduced crossed product} $A \rtimes \Ga$ is the $C^*$-subalgebra of $\B L\big( \ell^2(\Ga,A)\big)$ generated by the elements $\pi_\alpha (x) \lambda_g$ for $x \in A$ and $g \in \Ga$, where
\begin{equation}\label{eq:generator}
\pi_\al(x) \la_g( a \de_s) := \al_{gs}^{-1}(x) a \de_{gs} .
\end{equation}
We may identify the $C^*$-algebra $A$ with a $C^*$-subalgebra of the reduced crossed product $A \rtimes \Ga$ via the injective $*$-homomorphism $x \mapsto \pi_\al(x) \la_e$, which we shall often suppress. In the special case where $A = \B C$, we get that $\ell^2(\Ga) := \ell^2(\Ga,\B C)$ is a Hilbert space and the corresponding reduced crossed product agrees with the reduced group $C^*$-algebra $\Cred(\Ga) := \B C \rtimes \Ga$. We apply the notation $\B C \Ga \su \Cred(\Ga)$ for the \emph{group algebra} which is the smallest $*$-subalgebra containing $\la_g$ for all $g \in \Ga$. On the reduced group $C^*$-algebra we single out the tracial state $\tau \colon \Cred(\Ga) \to \B C$ which is defined by the formula $\tau(z) := \inn{\de_e,z \de_e}$. \\

The reduced crossed product $A \rtimes \Ga$ carries a coaction of $\Gamma$. This coaction induces a map $\delta \colon A \rtimes \Ga \to (A \rtimes \Ga) \ot_{\mathrm{min}} \Cred (\Ga)$ given by
\begin{equation}\label{eq:crossed-product-coaction}
\delta(\pi_\alpha(x) \la_g) = \pi_\alpha(x)\la_g \otimes \la_g .
\end{equation}
To see that  $\de$ is indeed a well-defined,  unital $*$-homomorphism, we represent the minimal tensor product $(A \rtimes \Ga) \ot_{\mathrm{min}} \Cred (\Ga)$ faithfully on the Hilbert $C^*$-module $\ell^2(\Ga,A) \hot \ell^2(G)$ and introduce the unitary operator $W \in \B L\big( \ell^2(\Ga,A) \hot \ell^2(G) \big)$ given by the formula
\[
W( a \delta_s \otimes \delta_t) = a \delta_s \otimes \delta_{st}.
\]
For every $z \in A \rtimes \Ga$, the coaction can then be described via conjugation with $W$ as
\begin{equation}\label{eq:coaction-through-W}
\delta(z) = W (z \otimes 1) W^* .
\end{equation}
The tracial state $\tau \colon \Cred(\Ga) \to \B C$ can be promoted to a conditional expectation $E \colon A \rtimes \Ga \to A \rtimes \Ga$ defined by putting $E(z) := (1 \ot \tau) \de(z)$. The image of $E$ agrees with $A$, {viewed as a $C^*$-subalgebra of $A \rtimes \Ga$}.        
	
	\subsection{Compact quantum metric spaces}
	In this section we review the basics of Rieffel's theory of compact quantum metric spaces \cite{Rie:MSS, Rie:MSA, Rie:GHD, Rie:CQM}, as well as some of the more recent results within the theory. To align with current developments in non-commutative geometry revolving around convergence of spectral truncations \cite{walter:GH-convergence, walter-connes:truncations, CvS:Truncations, GvS:TolFin, walter-tori, Rie:Fourier-truncations, walter-estrada-cp-grps}, the results in the present section are formulated using the framework of operator systems, rather than Rieffel's original approach based on order unit spaces.

	\begin{definition}
		An \emph{operator system} $\X$ is a unital $*$-invariant subspace of a unital $C^*$-algebra $A$. The operator system $\X$ is \emph{complete} if it is closed in the $C^*$-norm inherited from $A$. 
	\end{definition}
	
	

        From now on we fix an operator system $\C X \su A$ and let $X$ denote its norm closure in $A$ so that $X$ is a complete operator system. The scalars $\B C$ are often identified with the subspace $\B C 1 \su \C X$ spanned by the unit in $\C X$. An element $x \in \X$ is \emph{positive} if it is positive as an element in the ambient unital $C^*$-algebra $A$. The state space of $\C X$ is denoted by $S(\C X)$ and we recall that a \emph{state} on $\C X$ is a unital and positive map $\mu \colon \C X \to \CC$. The state space is equipped with the weak$^*$ topology and we record that $S(\C X)$ and $S(X)$ are homeomorphic through the restriction map. 

For two unital $C^*$-algebras $A$ and $B$, operator systems $\C X \su A$ and $\C Y \su B$ will be considered isomorphic if there is a unital completely positive isomorphism $\psi \colon \C X \to \C Y$ with a completely positive inverse. Notice here that the terminology completely positive means that $\psi$ induces a positive map from $M_n(\C X) \to M_n(\C Y)$ for all $n \in \B N$ by applying $\psi$ entrywise and considering $M_n(\C X)$ and $M_n(\C Y)$ as operator systems inside the unital $C^*$-algebra $M_n(A)$ and $M_n(B)$, respectively.
        
Let us now further equip the operator system $\C X$ with a seminorm $L \colon \X \to [0,\infty)$. The kernel of $L$ is defined as the subspace $\ker L:= \{x \in \X \mid L(x) = 0\}$. The following terminology was introduced by Rieffel in \cite{Rie:matricial-bridges}. 
	
\begin{definition}\label{def:Lipshitz-seminorm}
The seminorm $L$ is a called a \emph{slip-norm} if $\CC \subseteq \ker L$ and $L$ is $*$-invariant, meaning that $L(x^*) = L(x)$ for all $x \in \X$. 
\end{definition}

For the rest of this subsection we assume that our seminorm $L$ is a slip-norm.
	
\begin{definition}\label{def:Monge-Kantorovic-metric}
The \emph{Monge-Kantorovi\v{c} metric} $\rho_L$ on the state space $S(\X)$ is defined by 
\[
\rho_L(\varphi, \psi) := \sup \big\{\vert \varphi (x) - \psi(x) \vert \mid L(x) \leq 1  \big\}
\quad \mbox{for all } \varphi, \psi \in S(\C X).
\]
\end{definition}

Notice that the Monge-Kantorovi\v{c} metric is in fact an extended metric, meaning that $\rho_L$ satisfies all the usual properties of a metric except that infinite distances between states are allowed. Nonetheless, $\rho_L$ gives rise to a topology on $S(\C X)$ with basis consisting of all the metric open balls with finite radii.

\begin{definition}
The pair $(\X, L)$ is a \emph{compact quantum metric space} if the Monge-Kantorovi\v{c} metric $\rho_L$ metrizes the weak$^*$  topology on $S(\X)$.  In this case, $L$ is referred to as a \emph{Lip-norm}. 
\end{definition}
\begin{example}
The canonical commutative example, after which the above definition is modelled,  is obtained by considering a compact metric space $(M,\rho)$ and equipping {the unital $*$-subalgebra of Lipschitz functions $\Lip(M) \su C(M)$} with the seminorm which assigns the Lipschitz constant to a given {Lipschitz function}. The fact that this indeed defines a compact quantum metric space can be traced back to the work of Kantorovi\v{c} and Rubin\v{s}te\u{\i}n \cite{KaRu:FSE, KaRu:OSC}, and a more modern proof within the framework of compact quantum metric spaces can be found in \cite{Kaad2023}. 
\end{example}

The theory of compact quantum metric spaces is also inspired by, and strongly connected with, Connes' non-commutative geometry \cite{Con:NCG}, since a unital spectral triple $(\C A, H, D)$ naturally gives rise to a seminorm $L_D\colon \C A \to [0,\infty)$. {To explain how this works, recall that every element $a$ in the coordinate algebra $\C A$ preserves the domain of the selfadjoint unbounded operator $D$ (called the Dirac operator) and that the commutator
  \[
[D,a] \colon \T{Dom}(D) \to H
\]
extends to a bounded operator on the separable Hilbert space $H$. In particular, we get a derivation $d \colon \C A \to \B L(H)$ given by $d(a) := \T{cl}([D,a])$ and the relevant seminorm is defined by the formula
\[
L_D(a):=\|d(a)\|_\infty .
\]}

It is not always the case that $L_D$ gives rise to a compact quantum metric structure \cite{PutJul:subshifts, KN:finiteness}, but it is the case in many naturally occurring examples \cite{ChristRieffel2017, OzawaRieffel2005, AgKa:PSM, KaadMikkelsen2023, Christensen-Ivan:spectral-triples, BMR:DSS, FaLaLaPa:STN}. In \cite{BMR:DSS}, the following terminology {(suggested by Rieffel)} was introduced for this situation:


\begin{definition}\label{d:sms}
  {A \emph{spectral metric space} is a unital spectral triple $(\C A, H, D)$ satisfying that $(\C A, L_D)$ is compact quantum metric space.} 
\end{definition}

        
 Returning to the general setting of an operator system $\C X$ with a slip-norm $L$, we equip the quotient space $\C X/\B C$ with the quotient norm inherited from the $C^*$-norm on $A$ and denote the quotient map by $[\cdot] \colon \X \to \X /\CC$. The following theorem is Rieffel's characterisation of compact quantum metric spaces.

\begin{theorem}[{\cite[Theorem 1.8]{Rie:MSA}}]\label{thm:Rieffels-criterion}
The pair $(\X,L)$ is a compact quantum metric space if and only if the quotient map $[\cd]$ sends the $L$-unit ball $\{x\in \X \mid L(x)\leq 1\}$ to a totally bounded subset in $\X/\mathbb{C}$.
\end{theorem}	

\begin{remark}\label{rem:Rieffel-functional-version}
If $\sigma\colon \C X\to \B C$ is a bounded, unital functional, then we get an isomorphism of {normed spaces $\varphi\colon \C X \to  \C X/\B C  \oplus \B C$} by setting $\varphi(x)=([x], \sigma(x))$. The {subset}
\[
\mathcal{E}_{\sigma}:=\{x\in \X \mid L(x)\leq 1, \sigma(x)=0\} \su \C X
\]
is mapped onto $\{[x] \mid L(x)\leq 1\}\times \{0\}$ {via this isomorphism and \cref{thm:Rieffels-criterion} therefore shows that $\C E_\si \su \C X$ is totally bounded if and only if $(\C X, L)$ is a compact quantum metric space.}
\end{remark}

\begin{definition}\label{def:finite-diameter}
The pair $(\X, L)$ is said to have \emph{finite diameter} if there exists a constant $C \geq 0$ such that 
\[
\Vert [x] \Vert_{\X/\CC} \leq C \cdot L(x) \quad \mbox{for all } x \in \C X.
\]
\end{definition}

\begin{remark}	
Having finite diameter, in the sense of \cref{def:finite-diameter},  is equivalent with the (extended) metric $\rho_L$ assigning a finite diameter to the state space $S(\X)$; see \cite[Proposition 1.6]{Rie:MSA}. By connectedness and compactness of $S(\C X)$ in the weak$^*$  topology, it follows that  a compact quantum metric space automatically has finite diameter. Notice that connectedness is important here since the disjoint union of two compact metric spaces $M$ and $N$ is compact even though the distance between arbitrary points in $M$ and $N$ is deemed equal to infinity. 
\end{remark}

For the purposes of the present paper, we shall need the following convenient lemma, which we prove for lack of a good reference.

\begin{lemma}\label{l:directsum}
  Suppose that $(\C Y,K)$ is a compact quantum metric space and that $q \colon \C X \to \C Y$ and $\si \colon \C Y \to \C X$ are two unital maps satisfying 
  \begin{enumerate}
  \item $q \si = 1$;
  \item $q$ and $\si$ are slip-norm bounded and norm bounded;
  \item the subset $\big\{x \in \ker q \mid L(x) \leq 1 \big\} \su \C X$ is totally bounded.
  \end{enumerate}
Then it holds that $(\C X,L)$ is a compact quantum metric space.
\end{lemma}
\begin{proof}
  Let $\mu \colon \C Y \to \B C$ be a state. According to \cref{thm:Rieffels-criterion} and \cref{rem:Rieffel-functional-version}, it suffices to show that the subset
  \[
T := \big\{x \in \ker(\mu q) \mid L(x) \leq 1 \big\} \su \C X
\]
is totally bounded.
However, since both $q$ and $\si$ are slip-norm bounded there exists a constant $C > 0$ such that
\[
(1 - \si q)(T) \su \big\{x \in \ker q \mid L(x) \leq C \big\} \, \, \T{and } \, \, \,
(\si q)(T) \su \si\big( \big\{y \in \ker \mu \mid K(y) \leq C \big\} \big) .
\]
Since both of the right hand sides are totally bounded, we conclude that $T \su (1 - \si q)(T) + (\si q)(T)$ is totally bounded as well.
\end{proof}

Lastly, the following approximation result from \cite[Corollary 2.10]{KaadKyed2022} will play an important role in our proof of the main results.
\begin{proposition}\label{prop:CQMS-if-close-to-CQMS}
Suppose that for every $\varepsilon >0 $ there exists a compact quantum metric space $(\X_\varepsilon, L_\varepsilon)$ and unital linear maps $\Phi_\varepsilon \colon \X \to \X_\varepsilon$ and $\Psi_\varepsilon \colon \X_\varepsilon \to \X$ such that
\begin{enumerate}
\item $\Phi_\varepsilon$ is bounded for the slip-norms and $\Psi_\varepsilon$ is bounded for the operator norms;
\item The inequality $\Vert \Psi_\varepsilon \Phi_\varepsilon (x) - x \Vert_\infty \leq \varepsilon \cdot L(x)$ holds for all $x \in \C X$. 
\end{enumerate}
Then $(\X, L)$ is a compact quantum metric space. 
\end{proposition}

We are now in position to begin our investigation of compact quantum metric {space} structures on crossed products.

\section{Quantum metrics on crossed products}\label{sec:qm-on-crossed}
Throughout this section we let $A$ be a unital $C^*$-algebra equipped with an action $\al$ of a countable discrete group $\Ga$. We shall moreover consider an operator system $\C X \su A$ and suppose that $\al_g(\C X) = \C X$ for all $g \in \Ga$. 

For a subset $S \su \Ga$ satisfying that $e \in S$ and $S^{-1} = S$, we define the operator system $\C X \rtimes_{\alg} S$ as the span of the elements 
\[
\pi_\al(x) \la_g 
\quad \T{for all } x \in \C X \T{and } g \in S 
\]
inside $A \rtimes \Ga$.  
We apply the notation $X \rtimes S \su A \rtimes \Ga$ for the complete operator system obtained as the norm closure of $\C X \rtimes_\alg S$.

\subsection{Quantum metrics on subsystems of finite support}\label{ss:quametsubsys}
We recall from \cref{ss:cross} that $E \colon A \rtimes \Ga \to A$ denotes the conditional expectation defined by $E(z) := (1 \ot \tau) \de$. For each $g \in \Ga$, we then define the linear map $E_g \colon A \rtimes \Ga \to A$ by putting $E_g(z) := E( z \cd \la_g^{-1})$. Notice that $E_g$ is a norm contraction and that $E_g$ vanishes on $A$ as soon as $g \neq e$; we are here, as usual, identifying $A$ with the $C^*$-subalgebra $\pi_\al(A) \su A \rtimes \Ga$.

\begin{proposition}\label{prop:Berezin-image-is-CQMS}
{Let $S \su \Ga$ be a finite subset satisfying that $e \in S$ and $S^{-1} = S$.} Suppose that $(\C X,K)$ is a compact quantum metric space {and that $L$ is a slip-norm on $\C X \ralg S$.} If there exists a constant $C \geq 0$ such that the inequalities
  \[
  \| z - E(z) \|_\infty \leq C \cd L(z) \, \, \mbox{and } \, \, \, K( E_g(z) ) \leq C \cd L(z)
  \]
  are satisfied for all $z \in \C X \rtimes_{\alg} S$ and $g \in S$. Then the pair $\big( \C X \rtimes_{\alg} S, L \big)$ is a compact quantum metric space. 
\end{proposition}
\begin{proof}
Remark first of all that $\ker L = \B C$. Indeed, if $z \in \ker L$, then it follows from our assumptions that $z = E(z)$ and $K( E(z)) = 0$. Since  $(\C X,K)$ is assumed to be a compact quantum metric space, $\ker K = \B C$ and we may conclude that $z \in \B C$.\\
Fix a state $\mu \colon A \to \B C$ and consider the unital bounded operator $V_S \colon A \rtimes \Ga \to \Cred(\Ga)$ defined by 
  \[
V_S(z) := \sum_{g \in S} \la_g \mu( E_g(z) ) .
\]
We record that $V_S$ induces a surjective unital linear map $V_S \colon \C X \rtimes_{\alg} S \to \B C S$. {Moreover, let $\io \colon \Cred(\Ga) \to A \rtimes \Ga$ denote the injective unital $*$-homomorphism sending $\la_g$ to $\pi_\al(1) \la_g$ and record that $\io$ restricts to a right inverse $\io \colon \B C S \to \C X \ralg S$ of $V_S \colon \C X \rtimes_{\alg} S \to \B C S$.} The strategy is now to apply \cref{l:directsum} with $q=V_S$ and $\sigma = \io$ and we therefore verify the relevant conditions here below.\\
Let us identify $\B C S$ with an operator subsystem of $\C X \rtimes_{\alg} S$ via the {inclusion $\io$.} Applying the finite-dimensionality of $\B C S$ and the observation that $\ker L = \B C$ we then obtain that $(\B C S, L)$ is a compact quantum metric space. Regarding the boundedness condition, the only thing that needs proof is the fact that $V_S$ is slip-norm bounded. To see this, put $M:=\max\{L(\la_g) \mid g\in S\}$ and let $z \in \C X \rtimes_{\alg} S$ be given.
Applying our assumptions, we then get the estimates
\[
\begin{split}
  L\big( V_S(z) \big)
  & \leq \sum_{g \in S \sem \{e \}} L(\la_g) \cd \big| \mu(E_g(z))\big| 
\leq M \cd \sum_{g \in S \sem \{e\}} \| E_g(z) \|_\infty \\
& \leq M \cd (|S| - 1) \cd \| z - E(z) \|_\infty
\leq M \cd (|S| - 1) \cd C \cd L(z),
\end{split}
\]
establishing that $V_S \colon \C X \rtimes_{\alg} S \to \B C S$ is slip-norm bounded.\\
It therefore only remains to show that the subset
\[
R := \big\{z \in \ker V_S \mid L(z) \leq 1 \big\}
\su \C X \rtimes_{\alg} S
\]
is totally bounded. To verify this, remark that the set $T := \big\{x \in \C X \mid K(x) \leq C \T{ and } \mu(x) = 0 \big\}$ is totally bounded since $(\C X,K)$ is a compact quantum metric space; see \cref{rem:Rieffel-functional-version}. For every $z \in R$ and $g \in S$ it moreover holds that $E_g(z) \in T$ and hence that
\[
z = \sum_{g \in S} \pi_\al(E_g(z)) \cd \la_g \in \sum_{g \in S} \pi_\al(T) \cd \la_g.
\]
This shows that $R$ is totally bounded since $R$ is contained in the totally bounded set $\sum_{g \in S} \pi_\al(T) \cd \la_g$.
\end{proof}

Having now obtained a criterion for when finitely supported operator subsystems $\C X\rtimes_{\alg} S$  are compact quantum metric spaces, our next aim is to provide suitable maps from $\C X\rtimes_{\alg}  \Gamma$ into such {finitely supported subsystems. The end goal is to} deduce that $\C X\rtimes_{\alg}  \Gamma$ itself is a compact quantum metric space via an application of Proposition \ref{prop:CQMS-if-close-to-CQMS}. 

\subsection{The Berezin transform for crossed products}\label{ss:Berezin}
The  Berezin transform has played a pivotal role in the theory of compact quantum metric spaces since its first appearance in Rieffel's foundational paper \cite{Rie:MSG}. Most recently, it has been generalised to the setting of $q$-deformed spaces, being an essential tool in analysing the compact quantum metric structures of the Podle{\'s} sphere \cite{AKK:Podcon}, quantum $SU(2)$ \cite{KaadKyed2022} and the quantum projective spaces \cite{KaadMikkelsen2023}. The aim of the present subsection is to show how one may obtain a version of the Berezin transform for crossed products and apply it to investigate {quantum metrics in the situation where the group in question} is amenable. 
\medskip

Let us for a while fix a finite non-empty subset $F\su \Gamma$ and consider the corresponding unit vector
\begin{equation}\label{eq:unitvector}
\xi_F := \frac{1}{\vert F \vert^{1/2}} \sum_{s \in F} \delta_s 
\end{equation}
in the Hilbert space $\ell^2(\Ga)$. The unit vector $\xi_F$ then yields the vector state $\chi_F \colon \Cred (\Ga) \to \CC$ given by $\chi_F(z) := \inn{\xi_F , z \xi_F}$. Notice that in the special case where $F = \{e\}$, then $\chi_F$ agrees with the tracial state $\tau$. Recall, from \cref{ss:cross}, that the notation $\delta \colon A \rtimes \Ga \to (A \rtimes \Ga) \otimes_{\mathrm{min}} \Cred(\Ga)$ refers to the dual coaction, which we now use to define a Berezin transform in our setting. 

\begin{definition}\label{def:Berezin-transform-on-crossed-product}
The \emph{Berezin transform} associated to the finite non-empty subset $F \su \Ga$ is the unital positive map $\beta_F \colon A \rtimes \Ga \to A \rtimes \Ga$ given by
\begin{equation*}
\beta_F (z) = (1 \otimes \chi_F)  (\delta(z)) .
\end{equation*}
\end{definition}
A straightforward calculation shows that for  $z = \sum_{g \in \Ga} \pi_\al(x_g) \la_g \in A \rtimes_{\alg} \Ga$ we have the formulae
\begin{equation}\label{eq:berezin}
  \beta_F(z) = \sum_{g \in \Ga} \pi_\al(x_g) \la_g \cd \chi_F(\la_g)
  = \sum_{g \in \Ga} \frac{\vert F \cap gF \vert}{\vert F \vert} \pi_\al(x_g) \la_g.
\end{equation}
Let us define the finite subset $S := F \cd F^{-1} = \big\{s t^{-1} \mid s,t \in F \big\}$ and record that $e \in S$ and that $S = S^{-1}$. 

\begin{lemma}\label{l:imaber}
  We have the identities $\be_F( X \rtimes \Ga) = X \rtimes_\alg S$ and $\be_F(\C X \rtimes_\alg \Ga) = \C X \rtimes_\alg S$.
\end{lemma}
\begin{proof}
  Let first $g \in \Ga$ be given and record that
\[
\chi_F(\la_g) = 0 \lrar \vert F \cap gF\vert = 0 \lrar g \notin S
\]
Consulting the formulae in \eqref{eq:berezin} we then obtain that
\[
X \rtimes_\alg S \su \be_F( X \rtimes \Ga) \, \, \T{and } \, \, \, \be_F(\C X \rtimes_\alg \Ga) = \C X \rtimes_\alg S .
\]
The remaining inclusion follows by remarking that the finiteness of $F$ implies that $X \rtimes_{\alg} S$ is a norm-closed subspace of $X \rtimes \Ga$.
\end{proof}

For the rest of this subsection, we shall assume that $\Ga$ is \emph{amenable}. Letting $\epsilon\colon \CC\Ga \to \CC$ denote the counit defined by the formula $\epsilon(\la_g) = 1$ for all $g \in \Ga$, the amenability assumption on $\Ga$ amounts to demanding that $\epsilon$ extends to a character on $\Cred(\Gamma)$.  Amenability of $\Gamma$ may be characterised in a multitude of ways   \cite[Theorem 2.6.8]{Brown-Ozawa}, the most important for our purposes being by the existence of a  \emph{Følner sequence} $(F_n)_{n = 1}^\infty$; see \cite[Section 2.6]{Brown-Ozawa}. We recall here that each $F_n \su \Ga$ is a finite non-empty subset such that \begin{itemize}
\item[(i)] For every $g \in \Ga$ there exists an $N \in \nn$ such that $g \in F_n$ for all $n \geq N$ and;
\item[(ii)] For every $g \in \Ga$ it holds that $\lim_{n \to \infty} |F_n \cup g F_n|/|F_n| = 1 = \lim_{n \to \infty} |F_n \cap g F_n|/|F_n|$.
\end{itemize}  
For each $n \in \N$, we have the state $\chi_{F_n}$ and the Berezin transform $\beta_{F_n}$ from \cref{def:Berezin-transform-on-crossed-product}, which we will denote by $\chi_n$ and $\beta_n$, respectively.
The next result follows immediately from the properties of the Følner sequence together with the computation $\chi_n(\la_g) = |F_n \cap g F_n|/|F_n|$.

\begin{lemma}\label{lem:ptwise-convergence}
The sequence of states $(\chi_n)_{n = 1}^\infty$ converges to $\epsilon$ in the weak$^*$  topology.
\end{lemma}

Recall that $\C X \su A$ is an operator system with $\al_g(\C X) = \C X$ for all $g\in \Gamma$. Let us now fix a slip-norm $L$ on the operator system $\C X \rtimes_\alg \Ga \su A \rtimes \Ga$ and a slip-norm $L_{\Gamma}$ on the group algebra $\B C\Ga$, the latter viewed as an operator system inside $\Cred(\Ga)$. 
We let $\rho_{L_{\Gamma}}$ denote the corresponding Monge-Kantorovi\v{c} metric on the state space $S(\Cred(\Ga))$, which is homeomorphic to $S(\B C\Ga)$ via restriction.
Notice also that the dual coaction $\de$ maps the operator system $\C X \rtimes_\alg \Ga$ to an operator subsystem of the algebraic tensor product $(\C X \rtimes_\alg \Ga) \ot \B C\Ga$, {which sits inside the minimal tensor product $(A \rtimes \Ga) \otm \Cred(\Ga)$. We would also like to emphasise the subtlety that the quantity $\rho_{L_\Ga}(\chi_F,\epsilon)$, appearing in the statement of the next proposition, could in principle be infinite (making the corresponding estimate void).}

\begin{proposition}\label{prop:Berezin-norm-Lip-approx}
Suppose that $\Ga$ is amenable and that the inequality 
  \[
L_{\Gamma}( (\eta \ot  1)\de(z) ) \leq L(z)
\]
holds for all contractive functionals $\eta \colon X \rtimes \Ga \to \B C$ and all $z \in \C X \rtimes_\alg \Ga$. Then we have the estimate 
\[ 
\| \beta_F(z) - z \|_\infty \leq \rho_{L_{\Gamma}}(\chi_F, \epsilon) \cd L(z),
\]
for every $z \in \C X \rtimes_\alg \Ga$ and every finite non-empty subset $F \su \Ga$.
\end{proposition}
\begin{proof}
  Let $z \in \C X \rtimes_\alg \Ga$ be given and let $F \su \Ga$ be a finite non-empty subset. Consider an arbitrary contractive functional $\eta \colon X \rtimes \Ga \to \B C$. It then suffices to show that
  \[
\big| \eta\big( \beta_F(z) - z \big) \big| \leq \rho_{L_{\Gamma}}(\chi_F, \epsilon) \cd L(z) .
\]
Using that $( 1 \otimes \epsilon)\delta(z)=z$,  this estimate is a consequence of the following computation:
\[
\begin{split}
  \big| \eta\big(\be_F(z) - z \big)\big|
& = \big| \big(\eta \otimes (\chi_{F_n} - \epsilon) \big)\delta(z)  \big|
= \big| (\chi_F - \epsilon)\big( (\eta \otimes  1) \delta(z)  \big) \big| \\
&\leq \rho_{L_{\Gamma}}(\chi_F , \epsilon) \cdot L_{\Gamma} ((\eta \otimes  1) \delta(z)) 
\leq \rho_{L_{\Gamma}}(\chi_F , \epsilon) \cdot L(z) . \qedhere
\end{split}
\]
\end{proof}

\subsection{Constructing quantum metrics on crossed products}
 The aim of this short section is to prove the main technical result of the present paper, which provides criteria ensuring that  crossed products with amenable groups are compact quantum metric spaces. In the sections to follow, we shall see how our main theorems from the introduction follow from this result.

\begin{theorem}\label{thm:mainA}
  Suppose that $\Ga$ is an amenable, countable discrete group and that $(\C X, L_\X)$ and $(\B C\Ga, L_{\Gamma})$ are compact quantum metric spaces.
  Suppose moreover that the action of $\Ga$ on the unital $C^*$-algebra $A$ preserves the operator system $\C X \su A$. If $L$ is a slip-norm on $\C X \rtimes_\alg \Ga$ satisfying the following conditions:
  \begin{enumerate}
  \item The Berezin transform $\be_F \colon \C X \rtimes_\alg \Ga \to \C X \rtimes_\alg \Ga$ is slip-norm bounded for every finite non-empty subset $F \su \Ga$.
  \item The linear map $(\eta \ot 1) \de \colon \C X \rtimes_\alg \Ga \to \B C \Ga$ is a slip-norm contraction for every contractive functional $\eta \colon X \rtimes \Ga \to \B C$.
  \item The linear map $E_g \colon \C X \rtimes_\alg \Ga \to \C X$ is slip-norm bounded for every $g \in \Ga$.
  \end{enumerate}
Then the pair $\big( \C X \rtimes_\alg \Ga, L \big)$ is a compact quantum metric space.
\end{theorem}
\begin{proof}
Fix a Følner sequence $(F_n)_{n = 1}^\infty$ and the corresponding sequence of states $(\chi_n)_{n = 1}^\infty$ and Berezin transforms $(\beta_n)_{n = 1}^\infty$, constructed in \cref{ss:Berezin}. We aim to apply \cref{prop:CQMS-if-close-to-CQMS}, so let $\varepsilon >0$ be given.

Since $(\B C\Ga,L_{\Gamma})$ is assumed to be a compact quantum metric space, the Monge-Kantorovi\v{c} metric $\rho_{L_{\Gamma}}$ metrizes the weak$^*$  topology and by \cref{lem:ptwise-convergence} we may therefore choose an $n\in \N$ such that $\rho_{L_{\Gamma}} (\chi_n , \epsilon) \leq \varepsilon$. 
Using the language of \cref{prop:CQMS-if-close-to-CQMS}, we set $(\C X_\varepsilon, L_\varepsilon) = ( \X \rtimes_\alg S_n, L)$ where $S_n := F_n \cd F_n^{-1}$. By \cref{l:imaber} we may moreover consider the unital linear map $\Phi_\ep = \be_n \colon \X \rtimes_\alg \Ga \to \X \rtimes_\alg S_n$ and the unital linear map $\Psi_\ep \colon \X \ralg S_n \to \X \ralg \Ga$ given by inclusion. It follows from our assumptions that $\Phi_\ep$ is slip-norm bounded and, being an isometry, $\Psi_\ep$  is clearly norm-bounded. 

According to \cref{prop:CQMS-if-close-to-CQMS}, we need to show that $( \X \rtimes_\alg S_n, L)$ is a compact quantum metric space, {and to this end we apply} \cref{prop:Berezin-image-is-CQMS}. Because of our assumptions, it suffices to find a constant $C \geq 0$ such that $\| z - E(z) \|_\infty \leq C \cd L(z)$ for all $z \in \C X \ralg \Ga$. But this is a consequence of \cref{prop:Berezin-norm-Lip-approx} and the observation that $\chi_{\{e\}} = \tau$ so that 
\[
\| z - E(z) \|_\infty = \| z - \be_{\{e\}}(z) \|_\infty \leq \rho_{L_{\Gamma}}(\tau,\epsilon) \cd L(z).
\]
  {Notice here that the quantity $\rho_{L_{\Gamma}}(\tau,\epsilon)$ is indeed finite since $(\B C\Ga,L_\Ga)$ is a compact quantum metric space, implying that the diameter of the state space is finite (with respect to $\rho_{L_\Ga}$). By \cref{prop:Berezin-image-is-CQMS}, we may conclude that $( \X \rtimes_\alg S_n, L)$ is a compact quantum metric space.}
  
{To verify the remaining condition of \cref{prop:CQMS-if-close-to-CQMS}, we apply \cref{prop:Berezin-norm-Lip-approx} one more time, obtaining the inequalities}
\[
\| \Psi_\ep \Phi_\ep(z) - z\|_\infty = \| \beta_n(z) - z \|_\infty
\leq \rho_{L_{\Gamma}}(\chi_n, \epsilon) \cd L(z)
\leq \varepsilon \cd L(z)
\]
for all $z \in \C X \ralg \Ga$.

Since $\ep >0$ was arbitrary, it follows from \cref{prop:CQMS-if-close-to-CQMS} that $(\C X \ralg \Ga, L)$ is a compact quantum metric space. 
\end{proof}

\section{Slip-norms from length functions}\label{sec:slip-from-length}
{In this section we study slip-norms coming from a length function on the fixed countable discrete group $\Ga$. More precisely, assuming that $A$ is a unital $C^*$-algebra which is equipped with an action of $\Ga$, we shall see that a matrix length function $l \colon \Ga \to M_n(\B C)$ gives rise to a slip-norm $L_l$ on the algebraic crossed product $A \ralg \Ga$. In fact, if such a length function is proper we show how to construct an unbounded Kasparov module from our data and this unbounded Kasparov product then yields the relevant slip-norm in a canonical way. After these preliminary considerations we analyse the slip-norm $L_l$ on $A \ralg \Ga$ in more detail, culminating in a proof of \cref{thm:main-theorem-crossed}. Along the way, we shall also give a tentative definition of a bivariant spectral metric space based on our investigations of the pair $(A \ralg \Ga, L_l)$ and the characterisation of compact quantum metric spaces given in \cite[Theorem 3.1]{Kaad2023}; see \cref{r:bivariant} at the end of \cref{ss:slipslice}.}

\subsection{Unbounded Kasparov modules and slip-norms from length functions}\label{ss:length}
One of Connes' many insights in his seminal paper \cite{Con:CFH} is that {proper} length functions on countable discrete groups provide natural examples of spectral triples. The precise requirements for a length function on countable discrete groups {vary} a bit in the existing literature \cite{Con:CFH, HSWZ:STC, Klisse2023, ChristRieffel2017}, and to ensure optimal flexibility we choose to allow for matrix-valued length functions. As an example of this flexibility, notice that \cite[Remark 2.15]{HSWZ:STC} clarifies that their iteration procedure for crossed products with $\B Z$ can be implemented by allowing for matrix-valued length functions on $\B Z^d$ (for $d$ iterations). After providing the formal definiton, we further illustrate how the matrix valued setting allows for a description of the classical geometry of the 2-torus by means of a length function.
Throughout this subsection, $\Ga$ will be a countable discrete group and $n \in \B N$ is a fixed matrix size. 


\begin{definition}\label{d:length}
A function $l\colon \Gamma \to M_n(\B C)$ is called a \emph{matrix length function} {if the following hold:}
\begin{enumerate}
\item $l(g)$ is selfadjoint for all $g \in \Ga$.
\item $l(g) = 0$ if and only if $g = e$.
\item For all $g\in \Gamma$, the map $\varphi_g\colon \Gamma \to M_n(\B C)$ defined as $\varphi_g(s)=l(s)-l(g^{-1}s)$ is bounded.
\end{enumerate}
A matrix length function $l$ is called \emph{proper} if it satisfies that
\begin{enumerate}\setcounter{enumi}{3}
\item The map $(l^2 + 1)^{-1} \colon \Ga \to M_n(\B C)$ vanishes at infinity. 
\end{enumerate}
Letting $\ga \in M_n(\B C)$ be a selfadjoint unitary element, we say that a matrix length function $l$ is \emph{graded by $\gamma$}  if $l(g)$ anti-commutes with $\ga$ for all $g \in \Ga$.
\end{definition}

The prime example of such a (matrix) length function arises when $\Ga$ is finitely generated and $l$ agrees with the word length with respect to a choice of finite generating set. {However, to see the relevance of matrix length functions as opposed to scalar valued length functions consider the group $\B Z^2$ together with the proper length function $l \colon \B Z^2 \to M_2(\B C)$ given by $l(n,m) := \begin{pmatrix}0 & n + im \\ n - im & 0\end{pmatrix}$. This proper length function is moreover graded by the selfadjoint unitary element $\ga = \begin{pmatrix}1 & 0 \\ 0 & - 1\end{pmatrix}$. Following the constructions in the present subsection we get a graded unital spectral triple on the reduced group $C^*$-algebra $\Cred(\B Z^2)$. In this example, we may identify the corresponding Hilbert space $\ell^2(\B Z^2)^{\op 2}$ with the Hilbert space $L^2(\B T^2)^{\op 2}$, where each summand consists of $L^2$-functions on the $2$-torus. Under this identification, it is not hard to see that the Dirac operator coming from the matrix length function agrees with the classical Dirac operator $\begin{pmatrix} 0 & -i\pa/\pa \te_1 + \pa/\pa \te_2 \\ -i \pa/\pa \te_1 - \pa/\pa \te_2 & 0\end{pmatrix}$ for the $2$-torus (up to passing to the selfadjoint closure). Of course, we could also apply a matrix length function on $\B Z^k$ for $k > 2$ to reconstruct the classical Dirac operator for the $k$-torus as well.} 
\medskip

Let us from now on assume that $l \colon \Gamma \to M_n(\B C)$ is a proper matrix length function and that $\Ga$ acts via $*$-automorphisms on a unital $C^*$-algebra $A$. As in \cref{ss:cross}, we denote the action by $\al$ and the corresponding reduced crossed product by $A \rtimes \Ga$.
We shall see how to construct an unbounded Kasparov module which applies the proper matrix length function to relate the reduced crossed product $A \rtimes \Ga$ and the base $A$. This essentially follows the original construction due to Connes.
For the convenience of the reader, we start out by reviewing the definition of an unbounded Kasparov module from \cite{BaJu:TBK}. To this end, recall that an unbounded operator $D \colon \R{Dom}(D) \to E$ acting on a Hilbert $C^*$-module $E$ is called \emph{selfadjoint and regular} if $D$ is symmetric and $D + i$ and $D - i$ are {surjective; see \cite[Chapter 9]{Lance1995} for more details on these concepts.}  

\begin{definition} \label{def:unbnd-kasparov-module}
  Let $A$ and $B$ be unital $C^*$-algebras. {An} \emph{unbounded Kasparov module} from $B$ to $A$ is given by a triple $(\C B,E,D)$ where:
  \begin{itemize}
  \item[(i)] $\C B \su B$ is a norm-dense unital $*$-subalgebra of $B$.
  \item[(ii)] $E$ is a countably generated $C^*$-correspondence from $B$ to $A$ where the associated $*$-homomorphism $\phi \colon B \to \B L(E)$ is unital and injective.
  \item[(iii)] $D \colon \R{Dom}(D) \to E$ is an unbounded selfadjoint and regular operator acting on $E$.
  \end{itemize}
  This data is subject to the following two conditions:
  \begin{enumerate}
  \item For every $b \in \C B$ it holds that $\phi(b)$ preserves the domain of $D$ and  the operator
$[D,\phi(b)] \colon \R{Dom}(D) \to E$ extends to a bounded adjointable operator $d(b) \colon E \to E$.
\item The resolvent $(D + i)^{-1} \colon E \to E$ is compact (in the sense of Hilbert $C^*$-module theory).
  \end{enumerate}
  We say that an unbounded Kasparov module $(\C B,E,D)$ is \emph{graded} if there exists a selfadjoint unitary element $\ga \in \B L(E)$ such that $\ga \phi(b) = \phi(b) \ga$ for all $b \in B$ and $D\ga = - \ga D$. 
\end{definition}

Returning to the  crossed product situation, we now put $B := A \rtimes \Ga$ and consider the norm-dense unital $*$-subalgebra $\C B$ defined as the smallest $*$-subalgebra of $B$ containing all the bounded adjointable operators $\pi_\al(x) \la_g$, for $x \in A$ and $g \in \Ga$, introduced in \eqref{eq:generator}. In other words, $\C B$ can be identified with the algebraic crossed product $A \rtimes_{\alg} \Ga$. 
Let us represent the reduced crossed product $A \rtimes \Ga$ on the countably generated Hilbert $C^*$-module $E := \ell^2(\Ga,A) \ot \B C^n$ by letting each $z \in A \rtimes \Ga$ {act as $\phi(z) := z \ot 1$}. In this fashion, $E$ becomes a $C^*$-correspondence from $A \rtimes \Ga$ to $A$.
In order to define the relevant unbounded selfadjoint and regular operator $D_l \colon \T{Dom}(D_l) \to E$ we first consider the compactly supported maps $C_c(\Ga,A)$ as a norm-dense right $A$-submodule of $\ell^2(\Ga,A)$. We then introduce the unbounded symmetric operator
\[
\C D_l \colon C_c(\Ga,A) \ot \B C^n \to E \ \text{given by } \ \C D_l( a \de_s \ot \xi) := a \de_s \ot l(s) \xi . 
\]
The assumption that $l(g)$ is selfadjoint for all $g \in \Ga$ entails that $\C D_l + \la i$ has dense image for all $\la \in \B R \sem \{0\}$. Indeed, for $\la \in \B R \sem \{0\}$ and any simple tensor of the form $a \de_s \ot \xi \in C_c(\Ga,A) \ot \B C^n$, we get that
$(\C D_l + \la i)( a \de_s \ot (l(s) + \la i)^{-1} \xi) = a \de_s \ot \xi$. This entails that the image of $\C D_l + \la i$ does in fact agree with $C_c(\Ga,A) \ot \B C^n$. We therefore get that the closure $D_l := \R{cl}(\C D_l)$ is an unbounded selfadjoint and regular operator on $E$.  Remark that if $l$ is graded by a selfadjoint unitary $\ga \in M_n(\B C)$ then we may promote $\ga$ to a selfadjoint unitary operator $1 \ot \ga$ acting on $E = \ell^2(\Ga,A) \ot \B C^n$. The selfadjoint unitary $1 \ot \ga$ then anticommutes with $D_l$ and commutes with $z \ot 1$ for all $z \in A \rtimes \Ga$.\\

Before proving that the data $(\C B,E,D_l)$ is an unbounded Kasparov module, we make a small observation regarding multiplication operators on $E$. 
Let $\psi \colon \Ga \to M_n(A)$ be a {bounded map with respect to the $C^*$-norm on $M_n(A)$. We may then define the multiplication operator $M_\psi \colon E \to E$ by the formula
\begin{equation}\label{eq:mult}
M_\psi( a \de_s \ot e_j) := \sum_{i = 1}^n \psi(s)_{ij} \cd a \de_s \ot e_i ,
\end{equation}
where $\psi(s)_{ij} \in A$ refers to the entry in position $(i,j)$ for the matrix $\psi(s)$. Considering the unital $C^*$-algebra $\ell^\infty(\Ga,M_n(A))$ consisting of all bounded maps from $\Ga$ to the unital $C^*$-algebra $M_n(A)$ we then record that the above multiplication operators yield an injective unital $*$-homomorphism $M : \ell^\infty(\Ga,M_n(A)) \to \B L(E)$. In particular, we get that $\| M_\psi \|_\infty = \sup_{g \in \Ga} \| \psi(g) \|_\infty$ for all $\psi \in \ell^\infty(\Ga,M_n(A))$.

For a finite subset $F \su \Ga$, observe that the finite dimensionality of $\B C^n$ entails that the projection $P_F \colon E \to E$ given by
\[
P_F(a \de_s \ot \xi) = \fork{ccc}{a \de_s \ot \xi & \T{for} & s \in F \\ 0 & \T{for} & s \notin F}
\]
is a compact operator (in the sense of Hilbert $C^*$-module theory). For $\psi \in \ell^\infty(\Ga,M_n(A))$, using this compactness observation in combination with the computation of operator norms given above then yields that $M_\psi$ is a compact operator if and only if $\psi$ belongs to the $C^*$-subalgebra $C_0(\Ga,M_n(A)) \su \ell^\infty(\Ga,M_n(A))$ consisting of functions vanishing at infinity.

In the following sections, we will primarily be interested in the case where $\psi$ takes values in $M_n(\B C)\subseteq M_n(A)$.

\begin{proposition}\label{p:lengthunb}
Suppose that $l \colon \Ga \to M_n(\B C)$ is a proper matrix length function and that $\Gamma$ acts on a unital $C^*$-algebra $A$. Then the triple $\big( A \rtimes_{\alg} \Ga, \ell^2(\Ga,A) \ot \B C^n, D_l\big)$ is an unbounded Kasparov module from $A \rtimes \Ga$ to $A$. Moreover, if $l$ is graded by $\ga$, then the unbounded Kasparov module is graded by $1 \ot \ga$. 
\end{proposition}
We sketch the well-known proof for convenience of the reader and since we will be needing the concrete formula \eqref{eq:commutator-with-Dl} {for the associated derivation $d_l \colon A \rtimes_\alg \Ga \to \B L(E)$.} We are going to suppress the unital $*$-homomorphism $\phi \colon z \mapsto z \ot 1$ which represents $A \rtimes \Ga$ on $E = \ell^2(\Ga,A) \ot \B C^n$.

\begin{proof}
For each $g \in \Ga$, consider the multiplication operator $M_{\varphi_g}$ associated with the bounded function $\varphi_g\colon \Gamma \to M_n(\B C)$ given by $\varphi_g(s)=l(s)-l(g^{-1}s)$. 
  Each generator $\pi_\al(x)\lambda_g$ of the algebraic crossed product $A \rtimes_\alg \Ga$ preserves the domain of $\C D_l$ and a direct computation verifies that $[\C D_l, \pi_\al(x)\lambda_g ](\xi) = M_{\varphi_g}\pi_\al(x) \lambda_g(\xi)$ for all elements $\xi \in \T{Dom}(\C D_l) = C_c(\Ga,A) \ot \B C^n$. Since $M_{\varphi_g}\pi_\al(x) \lambda_g$ is a bounded adjointable operator on $E$, a standard approximation argument shows that $\pi_\al(x)\lambda_g$ also preserves the domain of $D_l$ and that $[D_l, \pi_\al(x)\lambda_g ](\xi) =M_{\varphi_g} \pi_\al(x)\lambda_g(\xi)$ for all $\xi \in \T{Dom}(D_l)$. {In particular, we get the formula
\begin{align}\label{eq:commutator-with-Dl}
d_l( \pi_\al(x) \la_g ) = \T{cl}\big( [D_l, \pi_\al(x) \la_g]\big) = M_{\varphi_g} \pi_\al(x)\lambda_g .
\end{align}}
It follows that every $z$ in $A \rtimes_\alg \Ga$ preserves the domain of $D_l$ and that $[D_l, z]$ extends to a bounded adjointable operator, as claimed.
To finish the proof, we notice that the resolvent $(D_l + i)^{-1}$ agrees with the multiplication operator $M_{(l + i)^{-1}} \colon E \to E$. Since the matrix length function $l \colon \Ga \to M_n(\B C)$ is assumed to be proper we get that $(l + i)^{-1} \colon \Ga \to M_n(\B C)$ vanishes at infinity, and by the observations preceding this proposition we may thus conclude that $(D_l + i)^{-1}$ is compact.
\end{proof}

Applying the above proposition we obtain a seminorm $L_l \colon A \rtimes_\alg \Ga \to [0,\infty)$ by setting $L_l(z) := \| d_l(z)\|_\infty$, where $d_l(z)$ agrees with the closure of the commutator $[D_l,z]$, {which is a priori only defined on $\T{Dom}(D_l)$}. The next result shows that $L_l$ is a slip-norm, but that the pair $\big( A \rtimes_\alg \Ga, L_l \big)$ can never be a compact quantum metric space unless $A = \B C$. As always, we view $A$ as a $C^*$-subalgebra of $A \rtimes \Ga$ via the identification $x \mapsto \pi_\al(x) \la_e$.  

  \begin{proposition}\label{p:kernel}
The seminorm $L_l \colon A \rtimes_\alg \Ga \to [0,\infty)$ is a slip-norm with $\ker L_l = A$.
  \end{proposition}
  \begin{proof}
    It clearly holds that $L_l$ is a slip-norm with $A \su \ker L_l$; the $*$-invariance follows since $d_l(z^*) = - d_l(z)^*$ for all $z \in A \rtimes_\alg \Ga$. Suppose thus that a finite sum of the form $z = \sum_{g \in \Ga} \pi_\al(x_g) \la_g$ belongs to $\ker L_l$. Using the formula from \eqref{eq:commutator-with-Dl} we then see that
    \[
    0 = d_l(z)( \de_e \ot \xi)
    = \sum_{g \in \Ga} \al_g^{-1}(x_g) \de_g \ot {l(g)} \xi
\]
for all $\xi \in \B C^n$. By forming the inner product with vectors of the form $\delta_t\otimes \eta$ it follows that {$\inn{\eta, l(t)\xi}\cdot\alpha_t^{-1}(x_t)=0$} for all $\xi,\eta\in \B C^n$ and all $t\in \Gamma$.
Using condition $(2)$ from \cref{d:length}, it now follows that {$x_t = 0$ for all $t \in \Ga \sem \{e\}$, meaning that $z \in A$}. 
  \end{proof}


  Let us end this section by briefly discussing the special case where $A = \B C$. In this setting we get that the reduced crossed product $A \rtimes \Ga$ agrees with the reduced group $C^*$-algebra $\Cred(\Ga)$ and the algebraic crossed product $A \rtimes_\alg \Ga$ agrees with the group algebra $\B C \Ga$. Moreover, the unbounded Kasparov module from \cref{p:lengthunb} becomes a unital spectral triple {on} $\Cred(\Ga)$ of the form $\big( \B C \Ga, \ell^2(\Ga) \ot \B C^n, D_l \big)$.  It is therefore natural to ask if this unital spectral triple becomes a {spectral metric space} in the sense of \cref{d:sms}. This turns out to be a difficult question to address in full generality, but
there are interesting cases where it is known to be true. For example, if $\Gamma$ is a finitely generated group which is either word-hyperbolic or of polynomial growth, then the usual word-length function arising from a finite generating set gives rise to a spectral metric space; see  \cite{ChristRieffel2017} and \cite{OzawaRieffel2005}, respectively.
 To the best of our knowledge no counterexample is known.

 \subsection{The Berezin transform is a slip-norm contraction}\label{ss:Berezin-L_l-contractive}
{We continue in the setting where $l \colon \Ga \to M_n(\B C)$ is a proper matrix length function on our countable discrete group $\Ga$. As usual, we also consider an action of $\Ga$ on a unital $C^*$-algebra $A$. The point of this section is to investigate the behaviour of the Berezin transforms from \cref{def:Berezin-transform-on-crossed-product} with respect to the slip-norm $L_l \colon A \rtimes_\alg \Ga \to [0,\infty)$, which comes from the unbounded Kasparov module $(A \rtimes_\alg \Ga, \ell^2(\Ga,A) \ot \B C^n,D_l)$.}

Let us fix a finite non-empty subset $F \su \Ga$ together with the associated unit vector $\xi_F \in \ell^2(\Ga)$ defined in \eqref{eq:unitvector}. This unit vector then induces a contractive linear map
\[
T_F \colon \ell^2(\Ga,A) \ot \B C^n \to ( \ell^2(\Ga,A) \ot \B C^n) \hot \ell^2(\Ga) \ \text{ given by } \ T_F(\eta) := \eta \ot \xi_F,
\]
and hence an associated unital contraction
\[
1 \ot \chi_F \colon \B L\big( (\ell^2( \Ga,A) \ot \B C^n) \hot \ell^2(\Ga)\big) \to \B L( \ell^2(\Ga,A) \ot \B C^n) 
\]
defined by the formula $(1 \ot \chi_F)(R) := T_F^* R T_F$. {Note that by viewing the minimal tensor product} $\B L(\ell^2( \Ga,A)\ot \B C^n)\otimes_{\min} \Cred(\Gamma)$ {as a unital $C^*$-subalgebra} of $\B L\big( (\ell^2( \Ga,A)\ot \B C^n) \hot \ell^2(\Ga)\big)$, the map $1 \ot \chi_F$ agrees with the usual slice map on the right tensor leg, thus justifying the notation.

We also introduce the unitary operator $W$ on $(\ell^2(\Ga,A) \ot \B C^n) \hot \ell^2(\Ga)$ given by the formula
\[
W(a\de_s \ot \xi \ot \de_t) := a \de_s \ot \xi \ot \de_{st}, 
\]
and, just as in \cref{ss:cross}, it can be verified that $W$ implements the dual coaction $\delta\colon A\rtimes \Gamma \to (A\rtimes \Gamma) \otimes_{\min} \Cred(\Gamma)$ from \eqref{eq:coaction-through-W}. {Notice here that $A\rtimes \Gamma$ is represented on $\ell^2(\Ga,A) \ot \B C^n$ via the unital injective $*$-homomorphism $z \mapsto z \ot 1$, so that the minimal tensor product is represented faithfully on $(\ell^2(\Ga,A) \ot \B C^n) \hot \ell^2(\Ga)$.} 

A combination of the above operations allows us to extend the Berezin transform $\be_F \colon A \rtimes \Ga \to A \rtimes \Ga$ from  \cref{def:Berezin-transform-on-crossed-product} to a unital contraction
\begin{equation}\label{eq:extendber}
  \begin{split}
  & \be_F \colon \B L( \ell^2(\Ga,A)\ot \B C^n) \to \B L( \ell^2(\Ga,A)\ot \B C^n) \ \text{ given by } \\
  & \be_F(R) := (1 \ot \chi_F) \big( W (R \ot 1) W^*\big) .
  \end{split}
\end{equation}
 It turns out that the extended Berezin transform is a homomorphism in a sense which we will now explain. To this end, consider a map $\psi \colon \Ga \to M_n(A)$ which is bounded with respect to the $C^*$-norm on $M_n(A)$ and recall the formula from \eqref{eq:mult} for the corresponding multiplication operator $M_\psi \in \B L\big( \ell^2(\Ga,A) \ot \B C^n \big)$. The precise statement now is the following:

  \begin{lemma}\label{l:module}
Let $\psi \colon \Ga \to M_n(A)$ be a bounded map. It holds that $\be_F( M_\psi R) = M_\psi \be_F(R)$ for all $R \in \B L\big( \ell^2(\Ga,A) \ot \B C^n \big)$.
\end{lemma}
\begin{proof}
This follows immediately from the definition of the extended Berezin transform and the two identities $W (M_\psi \ot 1) = (M_\psi \ot 1) W$ and $T_F^* (M_\psi \ot 1) = M_\psi T_F^*$. 
\end{proof}


As an application of \cref{l:module} we obtain that the Berezin transform is a contraction with respect to the slip-norm $L_l$.

\begin{proposition}\label{p:bersc}
For each finite, non-empty subset $F\subseteq \Gamma$, it holds that $L_l\big( \beta_F(z) \big) \leq L_l(z)$ for all $z \in A \ralg \Ga$.
\end{proposition}
\begin{proof}
{Let $z \in A \ralg \Ga$ be given.} Consider again the derivation $d_l\colon \C A\rtimes_{\alg} \Gamma \to \B L\big(\ell^2(\Gamma, A) \ot \B C^n\big)$ mapping $z$ to the closure of the commutator $[D_l,z]$. We are going to show that $\be_F d_l(z) = d_l \be_F(z)$, where the Berezin tranform on the left hand side is the extended version defined in \eqref{eq:extendber}. Since $L_l(z)=\|d_l(z)\|_\infty$ and $\beta_F$ is a norm contraction, the claimed result follows from this.

{To prove the desired identity, let $g \in \Ga$ and $x \in A$ and apply \eqref{eq:commutator-with-Dl} and \eqref{eq:berezin} together with \cref{l:module} to get that
  \[
  \begin{split}
  \be_F d_l( \pi_\al(x) \la_g) & = \be_F( M_{\varphi_g} \pi_\al(x) \la_g)
  = M_{\varphi_g} \pi_\al(x) \be_F(\la_g) \\
  & = M_{\varphi_g} \pi_\al(x) \la_g \cd \inn{\xi_F,\la_g \xi_F}
  = d_l( \pi_\al(x) \la_g \cd \inn{\xi_F,\la_g \xi_F})
  = d_l \be_F(\pi_\al(x) \la_g) .
  \end{split}
  \]
  The relevant identity $\be_F d_l(z) = d_l \be_F(z)$ now follows by linearity and the above computation.}
\end{proof}


\subsection{The slip-norm and slice maps}\label{ss:slipslice}
{Recall that the countable discrete group $\Ga$ has been equipped with a matrix length function $l \colon \Ga \to M_n(\B C)$ and that $\Ga$ acts on the unital $C^*$-algebra $A$.}

The aim of this subsection is to prove the following proposition which describes the behaviour of the slip-norm $L_l$ under slice maps. Note that in the inequality displayed in the proposition here below, we are strictly speaking referring to two different slip-norms: On the left hand side  $L_l \colon \B C \Ga \to [0,\infty)$ comes from the unital spectral triple $( \B C \Ga, \ell^2(\Ga)\ot \B C^n, D_l)$ and on the right hand side $L_l \colon A \ralg \Ga \to [0,\infty)$ comes from the unbounded Kasparov module $\big( A \ralg \Ga, \ell^2(\Ga,A)\ot \B C^n, D_l\big)$. Both of these ingredients were explained in \cref{ss:length}. However, applying the same notation is justified by recording that $L_l \colon \B C \Ga \to [0,\infty)$ can also just be viewed as a restriction of $L_l \colon A \ralg \Ga \to [0,\infty)$ to the unital $*$-subalgebra $\B C \Ga \su A \ralg \Ga$. The proof of the proposition requires some preparations and we present it at the end of this subsection. 

\begin{proposition}\label{p:slice}
  Suppose that $l \colon \Ga \to M_n(\B C)$ is a proper matrix length function and that $\eta \colon A \rtimes \Ga \to \B C$ is a norm contraction. Then we have the inequality 
  \[
 L_l\big( (\eta \otimes 1) \delta(z) \big) \leq L_l(z) \quad \mbox{for all } z \in A \ralg \Ga .
\]
\end{proposition}

{To} every $g \in \Ga$ we associate two unitary operators $\lambda_g$ and $\rho_g \colon \ell^2(\Ga,A)\ot \B C^n \to \ell^2(\Ga,A)\ot \B C^n$ defined by
\[ 
{\lambda_g (a \delta_s \ot \xi) = a\delta_{gs} \ot \xi \quad \text{and } \quad \rho_g (a \delta_s \ot \xi) = a \delta_{sg^{-1}} \ot \xi .}
\]
These two unitary operators determine the left regular and the right regular representation (respectively) of the {countable} discrete group $\Ga$ in the special case where $A = \B C$ and $n = 1$.

For every $t \in \Ga$, we define the essentially selfadjoint and regular unbounded operator $\C D_l^t \colon C_c(\Ga,A)\ot \B C^n \to \ell^2(\Ga,A)\ot \B C^n$ by the formula
\[
{\C D_l^t( a \delta_s \ot \xi) = a \delta_s \ot l(st) \cd \xi} 
\]
and note that $\C D_l^e$ agrees with the unbounded operator $\C D_l$ introduced in \cref{ss:length}. The closure of $\C D_l^t$ is denoted by $D_l^t$. We observe that $\rho_t D_l \rho_t^{-1} = D_l^t$, so $D_l^t$ and $D_l$ are unitarily equivalent. 
The {sequence} of unbounded operators $(\C D_l^t)_{t\in \Gamma}$ can be assembled into a single essentially selfadjoint and regular unbounded operator
{\[
\begin{split}
& \tilde{\C D}_l \colon C_c(\Ga,A)\ot \B C^n \ot C_c(\Ga) \to (\ell^2(\Ga,A)\ot \B C^n) \hot \ell^2(\Ga) \ \text{ given by } \\
  & \tilde{\C D}_l(a \delta_s \ot \xi \otimes \delta_t) := \C D_l^t(a\delta_s \ot \xi) \otimes \delta_t .
\end{split}
\]}
The closure of $\tilde{\C D}_l$ is denoted by $\tilde{D}_l$. We shall also need the essentially selfadjoint and regular unbounded operator
{\[
\begin{split}
& 1 \ot \C D_l \colon C_c(\Ga,A) \ot C_c(\Ga)\ot \B C^n \to \ell^2(\Ga,A) \hot ( \ell^2(\Ga)\ot \B C^n ) \ \text{ given by } \\ 
& (1 \ot \C D_l)(a \delta_s \otimes \delta_t \ot \xi) := a\delta_s \otimes \de_t \ot l(t)\cd \xi .
\end{split}
\]}
The closure of $1 \ot \C D_l$ is denoted by $1 \hot D_l$.

Consider the unitary operator
\[
\begin{split}
& W \colon ( \ell^2(\Ga,A)\ot \B C^n) \hot \ell^2(\Ga)  \to \ell^2(\Ga,A) \hot ( \ell^2(\Ga)\ot \B C^n) \ \text{ given by } \\
& W(a \de_s \ot \xi \ot \de_t) := a \de_s \ot \de_{st} \ot \xi .
\end{split}
\]
We represent the minimal tensor product $(A \rtimes \Ga) \otm \Cred(\Ga)$ faithfully on the Hilbert $C^*$-module $\ell^2(\Ga,A) \hot (\ell^2(\Ga)\ot \B C^n)$ using the defining representation for $A \rtimes \Ga$ and the faithful representation of $\Cred(\Ga)$ on $\ell^2(\Ga)\ot \B C^n$ given by $y \mapsto y \ot 1$. As in \eqref{eq:coaction-through-W}, we may still describe the dual coaction $\de \colon A \rtimes \Ga \to (A \rtimes \Ga) \otm \Cred(\Ga)$ via conjugation with $W$, meaning that
\begin{equation}\label{eq:coactW}
\de(z) = W (z \ot 1 \ot 1) W^* \quad \T{for all } z \in A \rtimes \Ga .
\end{equation}
The unitary operator $W$ also describes the relationship between $\tilde{D}_l$ and {$1 \hot D_l$} in so far that
\begin{equation}\label{eq:tildenon}
W \tilde{\C D}_l W^*(\ze) = (1 \ot \C D_l) (\ze)
\end{equation}
for all elements $\ze$ belonging to $C_c(\Ga,A) \ot C_c(\Ga)\ot \B C^n$.

The next lemma follows from the argument presented in the proof of \cref{p:lengthunb}. For this reason, we could also replace the core $C_c(\Ga,A)\ot \B C^n$ {for $D_l^t$ with the domain of $D_l^t$} in the statement of the lemma.

\begin{lemma}
Let $t \in \Ga$ and $z \in A \ralg \Ga$. The operator $z$ preserves the subspace $C_c(\Ga,A)\ot \B C^n \subseteq \ell^2(\Ga,A)\ot \B C^n$, and the commutator
\[
  [\C D_l^t , z] \colon C_c(\Ga,A)\ot \B C^n \to \ell^2(\Ga,A)\ot \B C^n
\]
extends to a bounded adjointable operator $d_l^t(z)$ on $\ell^2(\Ga,A)\ot \B C^n$.
\end{lemma}

Let $t \in \Ga$ and consider the bounded matrix-valued map $\varphi_g^t \colon \Ga \to M_n(\B C)$ defined by putting $\varphi_g^t(s) := l(st) - l(g^{-1} st)$. It is then convenient to record the explicit formula
\[
d_l^t( \pi_\al(x) \la_g) = \pi_\al(x) M_{\varphi_g^t} \la_g \quad \T{for all } x \in A \, \, , \, \, \, g \in \Ga .
\]
Notice also that $M_{\varphi_g^t} = \rho_t M_{\varphi_g} \rho_t^{-1}$ and hence that
\begin{equation}\label{eq:rhodlt}
\rho_t^{-1} d_l^t(\pi_\al(x) \la_g) \rho_t = \rho_t^{-1} \pi_\al(x) \rho_t M_{\varphi_g} \la_g .
\end{equation}
Let us finally introduce the isometric isomorphism
{\[
\begin{split}
 & U_t \colon \ell^2(\Ga,A)\ot \B C^n \to \ell^2(\Ga,A)\ot \B C^n \ \T{ given by } \\
 & U_t( a \de_s \ot \xi) := \al_t(a) \de_s \ot \xi .
\end{split}
\]}
We mention in passing that the {sequence $( U_t )_{t \in \Ga}$} yields an action of the {countable discrete} group $\Ga$ on the Hilbert $C^*$-module $\ell^2(\Ga,A)\ot \B C^n$, see e.g. \cite[\S 1.2]{Kas:EKK}, but we shall not be needing this fact here. The relationship between the isometric isomorphisms $U_t$ and $\rho_t$ can then be expressed by the identity
\begin{equation}\label{eq:yourho}
{U_t (z \ot 1 ) U_t^{-1} = \rho_t^{-1} (z \ot 1) \rho_t},
\end{equation}
which is valid for all elements $z \in A \rtimes \Ga$. 
	
\begin{lemma}\label{lemma:commutator-norm-t-independence}
Let $t \in \Ga$ and $z \in A \ralg \Ga$. It holds that
\[
\rho_t^{-1} d_l^t(z) \rho_t = U_t  d_l(z) U_t^{-1} 
\]
as operators on $\ell^2(\Ga,A)\ot \B C^n$. In particular, we have the identity $\| d_l^t(z) \|_\infty = \| d_l (z) \|_\infty$. 
\end{lemma}
\begin{proof}
We may restrict to the case where $z = \pi_\al(x) \la_g$ for some $x \in A$ and $g \in \Ga$. It follows from \eqref{eq:rhodlt} and \eqref{eq:yourho} that 
\[
\rho_t^{-1} d_l^t(z) \rho_t = \rho_t^{-1} \pi_\al(x) \rho_t M_{\varphi_g} \la_g = U_t \pi_\al(x) U_t^{-1} M_{\varphi_g} \la_g
= U_t d_l(z) U_t^{-1} .
\]
This proves the present lemma.
\end{proof}

{In the statement of the next lemma, we could replace the core $C_c(\Ga,A)\ot \B C^n \ot C_c(\Ga)$ for $\wit{D}_l$ with the domain of $\wit{D}_l$, but this is less important for our purposes.}

\begin{lemma}\label{lemma:norm-of-commutator-in-third-entry}
For every $z \in A \ralg \Ga$, the operator {$z \otimes 1 \ot 1$} preserves the dense $A$-submodule $C_c(\Ga,A)\ot \B C^n \ot C_c(\Ga) \su (\ell^2(\Ga,A)\ot \B C^n) \hot \ell^2(\Ga)$, and the commutator
\[
\big[\tilde{\C D}_l, {z \otimes 1 \ot 1}\big] \colon C_c(\Ga,A)\ot \B C^n \ot C_c(\Ga) \to (\ell^2(\Ga,A)\ot \B C^n) \hot \ell^2(\Ga)
\]
extends to a bounded adjointable operator $\tilde{d}_l(z)$ on $(\ell^2(\Ga,A)\ot \B C^n) \hot \ell^2(\Ga)$. Furthermore, we have the identity
\[
\big\| \tilde{d}_l(z) \big\|_\infty = \| d_l(z) \|_\infty .
\]
\end{lemma}
\begin{proof}
{A computation of the commutator in question shows that} the bounded adjointable operator $\tilde{d}_l(z)$ agrees with the diagonal operator given by
\[
\tilde{d}_l(z)(\ze \otimes \delta_t) = d_l^t(z)(\ze) \otimes \delta_t 
\]
for all $\ze \in C_c(\Ga,A)\ot \B C^n$ and $t \in \Ga$. The result now follows {from the operator norm identity in} \cref{lemma:commutator-norm-t-independence} and {standard properties of diagonal operators.}
\end{proof}

%
%

Recall the definition of the dual coaction $\delta : A \rtimes \Ga \to (A \rtimes \Ga) \otm \Cred(\Ga)$ from \cref{ss:cross}.

\begin{proposition}\label{prop:commutator-norm-equality}
For every $z \in A \ralg \Ga$ we have the identity
\[
\big\|  (1 \otimes d_l)(\delta(z))\big\|_\infty = \big\|  d_l(z) \big\|_\infty .
\]
\end{proposition}
\begin{proof}
  Remark that $(1 \ot d_l)(\de(z))$ does indeed make sense since $\de(z) \in (A \ralg \Ga) \ot \B C \Ga$. Letting $\ze$ belong to $C_c(\Ga,A) \ot C_c(\Ga)\ot \B C^n$ we apply \eqref{eq:coactW} and \eqref{eq:tildenon} to obtain the identities
  \[
  \begin{split}
    (1 \ot d_l)(\de(z))(\ze) & = [1 \ot \C D_l, W(z \ot 1 \ot 1) W^*](\ze) \\
    & = W [\tilde{\C D}_l,z \ot 1 \ot 1] W^*(\xi) = W\tilde{d}_l(z) W^*(\ze).
  \end{split}
  \]
The result of the proposition therefore follows from \cref{lemma:norm-of-commutator-in-third-entry}. 
\end{proof}

{We are now ready to prove the main result of this subsection, which shows that certain slice maps are slip-norm contractions.}

\begin{proof}[Proof of \cref{p:slice}]
  Let $z \in A \ralg \Ga$ be given. Recall first of all that the norm contraction $\eta \colon A \rtimes \Ga \to \B C$ extends to a norm contraction $\eta \ot 1$ on $(A \rtimes \Ga) \otm \B L(\ell^2(\Ga)\ot \B C^n)$; cf.~\cite[pg. 79]{Lance1995}. Applying  \cref{prop:commutator-norm-equality}, we therefore get that 
  \[
   L_l\big( (\eta \otimes 1) \delta(z) \big)  
  = \big\| (\eta \ot 1) (1 \ot d_l) \de(z) \big\|_\infty \leq \big\| (1 \ot d_l) \de(z) \big\|_\infty = L_l(z) . \qedhere
  \]
\end{proof}

\begin{remark}\label{r:bivariant}
In the case where $\Gamma$ is amenable, \cref{prop:Berezin-norm-Lip-approx} and \cref{p:slice} show that the Berezin transform satisfies the estimate
\begin{align}\label{eq:beta-F-approx}
\|\beta_F(z)-z\|_\infty \leq \rho_{L_l}(\chi_F, \epsilon) L_l(z) \quad \T{for all } z \in A\rtimes_{\alg}\Gamma
\end{align}
    {as soon as $F \su \Ga$ is a non-empty finite subset. Moreover, by \cref{l:imaber} we know that the image of the Berezin transform $\beta_F$ agrees with the finitely generated free $A$-module $A \ralg S$ where $S = F \cd F^{-1}$.}

    Now, choosing $F=\{e\}$, the state $\chi_F$ {agrees with the tracial state} $\tau$ on $\Cred(\Gamma)$ and $\beta_F$ agrees with the conditional expectation $E\colon A\rtimes \Gamma \to A \rtimes \Ga$ with image $A$. In this situation, the inequality \eqref{eq:beta-F-approx} therefore implies that the slip-unit ball $\{x\in A\rtimes_{\alg}\Gamma \mid L_l(x)\leq 1 \}$ projects onto a subset of diameter at most $2\cdot \rho_{L_l}({\tau}, \epsilon)$ in the quotient space $(A\rtimes \Gamma)/A$.

    If  $(\B C \Gamma, L_l)$ is furthermore assumed to be a compact quantum metric space we also have that $\lim_{n\to \infty}\rho_{L_l}(\chi_{F_n}, \epsilon)=0$ for a Følner sequence $(F_n)_n$ in $\Gamma$; see \cref{lem:ptwise-convergence}. {In this case, we also know that the distance $\rho_{L_l}(\tau, \epsilon)$ is in fact finite.}
    
    Since $\ker(L_l)=A$ {by \cref{p:kernel}} one may view the above properties {for the pair $(A \ralg \Ga, L_l)$} as an $A$-valued analogue of the characterisation of compact quantum metric spaces given in \cite[Theorem 3.1]{Kaad2023}. Under the assumption that $\Ga$ is amenable and $(\B C\Ga,L_l)$ is a compact quantum metric space, it is therefore tempting to view the unbounded Kasparov module $(A \ralg \Ga, \ell^2(\Ga,A) \ot \B C^n, D_l)$ as a  \emph{bivariant spectral metric space}. Let us highlight the properties for our unbounded Kasparov module which we view as essential in this respect (even though the codomain of $\Phi_\ep$ could perhaps be enlarged):
    \begin{enumerate}
    \item The image of the slip-unit ball $\{x\in A\rtimes_{\alg}\Gamma \mid L_l(x)\leq 1 \}$ in the quotient space $(A \rtimes \Ga)/ A$ is bounded (with respect to the quotient norm);
    \item For every $\ep > 0$ there exists a unital bounded $A$-linear map $\Phi_\ep \colon A \rtimes \Ga \to A \rtimes \Ga$ such that the image of $\Phi_\ep$ is a finitely generated projective module over $A$ and the norm-estimate
      \[
\| x - \Phi_\ep(x) \|_\infty \leq \ep \cd L_l(x)
\]
holds for all $x \in A \ralg \Ga$.
    \end{enumerate}
\end{remark}

\subsection{Proof of \cref{thm:main-theorem-crossed} and \cref{maincor:poly-growth-cor-cqms}}
With the tools above at our disposal, we may now prove  \cref{thm:main-theorem-crossed} and  \cref{maincor:poly-growth-cor-cqms} from the introduction.

{Let us recall that $\C X \su A$ is an operator system which is preserved by the action of the countable discrete group $\Ga$ on the unital $C^*$-algebra $A$. We are moreover fixing a proper matrix length function $l \colon \Ga \to M_n(\B C)$ and a slip-norm $L_{\C X} \colon \C X \to [0,\infty)$. Moreover, we recall from the beginning of \cref{ss:quametsubsys} that $E_g \colon A \rtimes \Ga \to A$ denotes the norm contraction defined by $E_g(z) := E( z \cd \la_{g^{-1}})$.}

\begin{proof}[Proof of \cref{thm:main-theorem-crossed}]
  {Suppose that $(\C X,L_{\C X})$ and $(\B C\Ga,L_l)$ are compact quantum metric spaces and that $\Ga$ is amenable. Let moreover $\nvert{\cdot}$ be an order-preserving norm on $C_c(\Gamma, \B C)$.}
  
  {To each $z$ belonging to the algebraic operator system crossed product $\C X \ralg \Ga$ we associate the compactly supported function $L_{\C X} \ci z \colon \Ga \to [0,\infty)$ defined by $(L_{\C X} \ci z)(g) := L_{\C X}( E_g(z))$. In particular, we obtain the seminorm $L_H \colon \C X \ralg \Ga \to [0,\infty)$ given by the formula $L_H(z) := \nvert{L_{\C X}\circ z}$. We also have the slip-norm $L_l$ on $\C X \ralg \Ga$ coming from our unbounded Kasparov module $\big( A \ralg \Ga, \ell^2(\Ga,A) \ot \B C^n, D_l \big)$.}
      
  The claim is now that the slip-norm
\[
{L(z)=\max\{L_l(z), L_H(z), L_H(z^*)\}}
\]
provides a compact quantum metric {space} structure on $\X\ralg \Gamma$. 

{We are going to apply \cref{thm:mainA} and verify the three relevant conditions, starting with condition $(1)$. Let $F \su \Ga$ be a finite non-empty subset.} By \cref{p:bersc}, the Berezin transform $\beta_F$ is a contraction for the slip-norm $L_l$ and we now verify that this is also the case for the remaining parts defining $L$. Let $z = \sum_{s\in \Gamma} \pi_{\alpha}(x_s) \lambda_s \in \X\rtimes_{\alg} \Gamma$ be given. {The formula in \eqref{eq:berezin} shows that
\[
\beta_F(z) = \sum_{s \in \Ga} \pi_\al(x_s) \la_s  \cd \chi_F(\la_s) .
\]
Since $\big| \chi_F(\la_g) \big| \leq 1$ it therefore follows that
\[
\big(L_{\C X} \ci \be_F(z)\big)(g) = L_{\C X}( x_g ) \cd \big| \chi_F(\la_g) \big| \leq (L_{\C X} \ci z)(g) .
\]
for all $g \in \Ga$. Since $\nvert{\cd}$ is assumed to be order-preserving we get that
\[
L_H( \be_F(z)) = \nvert{L_{\C X} \ci \be_F(z)} \leq \nvert{L_{\C X} \ci z } = L_H(z) .
\]
Since this inequality of course also holds for $z^*$ instead of $z$, we conclude that condition $(1)$ of  \cref{thm:mainA} is satisfied.}

Regarding condition $(2)$, \cref{p:slice} shows that this is already satisfied for the slip-norm $L_l$ on $\X\rtimes_{\alg} \Gamma$ and hence also for the bigger slip-norm $L$.

For condition $(3)$, fix a $g\in \Gamma$ {and a $z \in \X\rtimes_{\alg} \Gamma$. By definition of the compactly supported function $L_{\C X} \ci z$, we have that
\[
0 \leq L_{\C X}( E_g(z))\delta_g(s) \leq (L_{\C X} \ci z)(s)
\]
for all $s \in \Ga$. Since $\nvert{\cdot}$ is assumed order-preserving} we obtain the estimate
\[
L_\X(E_g(z))  =\frac{1}{\nvert{\delta_g}} \cdot \nvert{L_\X( E_g(z))\delta_g}\leq  \frac{1}{\nvert{\delta_g}}\cdot\nvert{L_{\C X}\circ z}
\leq \frac{1}{\nvert{\delta_g}} \cdot L(z),
\]
showing that $E_g$ is indeed slip-norm bounded. Hence all conditions in  \cref{thm:mainA}  are met, and we conclude that $(\X\rtimes_{\alg} \Gamma, L)$ is a compact quantum metric space.
\end{proof}

\begin{proof}[Proof of \cref{maincor:poly-growth-cor-cqms}]
{This follows from \cref{thm:main-theorem-crossed}} since groups of polynomial growth are amenable \cite[Example 2.6.6]{Brown-Ozawa} and their word length functions are known to yield compact quantum metric spaces \cite[Theorem 1.4]{ChristRieffel2017}. 
\end{proof}

\section{Spectral metrics on crossed products}\label{ss:specross}
We now turn our attention to  the important special case where the Lip-norm on the base algebra of the crossed product stems from a spectral triple, and our main goal is to prove 
\cref{introthm:spectral-metric-space} from the introduction. Throughout this section, we therefore let $(\C A,H,D)$  be a unital spectral triple of parity $q \in \{0,1\}$. With our conventions, the unital $*$-algebra $\C A$ is identified with a unital $*$-subalgebra of $\B L(H)$ and the corresponding norm closure is denoted by $A$, so that $A$ is a unital $C^*$-algebra. In the {graded} case (meaning $q = 0$) we denote the $\zz/2\zz$-grading operator by $\ga \colon H \to H$.
On top of this data, we consider a countable discrete group $\Ga$  which acts on $A$ via $*$-automorphisms and, as before, the action of a group element $g$ is denoted by $\al_g \colon A \to A$. We shall furthermore assume that $\Ga$ is equipped with a proper matrix length function $l \colon \Ga \to M_n(\B C)$ for some fixed matrix size $n \in \B N$; see \cref{d:length}. The parity $p \in \{0,1\}$ of the length function is defined to be zero in the graded case and one otherwise. {For $p = 0$ we denote the corresponding selfadjoint unitary by $\ga_l \in M_n(\B C)$.}

\subsection{Spectral triples on crossed products}\label{ss:triples-on-crossed}
Our unital spectral triple $(\C A,H,D)$ gives rise to a derivation $d \colon \C A \to \B L(H)$ given by the formula $d(a) := \T{cl}([D,a])$, meaning that $d(a)$ agrees with the closure of the commutator $[D,a] \colon \T{Dom}(D) \to H$. In order to promote the selfadjoint unbounded operator $D$ to (part of a) Dirac operator at the crossed product level, one needs to require more from the the action of $\Ga$. More precisely, it turns out that the following additional assumptions are needed, where the terminology is chosen to align with that from \cite{HSWZ:STC} and \cite{Klisse2023}.

\begin{definition}
The action of $\Ga$ is said to be \emph{smooth} if  $\al_g(\C A) = \C A$ for all $g \in \Ga$  and \emph{metrically equicontinuous} if, moreover, 
  \[
\big\{\big\| d( \al_g(a)) \big\|_\infty \mid g \in \Ga \big\} \su [0,\infty)
  \]
  is bounded for all $a\in \A$.
\end{definition}

That smoothness and metric equicontinuity are necessary additional assumptions was already realised in \cite{BMR:DSS} in the case where $\Gamma=\zz$, and later considered in the context of general countable discrete groups in \cite{HSWZ:STC} and in \cite{Pat:CST} for locally compact groups. See also \cite{HSWZ:STC} for a discussion of the role of metric equicontinuity for actions on classical spaces. The focus in \cite{HSWZ:STC} and  \cite{BMR:DSS} is on the situation where both the length function and the spectral triple are odd, but since we also aim to show that our constructions are compatible with the Kasparov product in $KK$-theory, we will need to pay additional attention to the three remaining combinations of parities.
The $KK$-theoretic point of view has two obvious advantages: Firstly, it establishes that the relevant unital spectral triple on the reduced crossed product $A \rtimes \Ga$ is constructed by a method which applies in many other situations. And secondly, the application of the unbounded Kasparov product in this context  clarifies the link with the internal Kasparov product in $KK$-theory, which helps with understanding the $K$-homology class of the unital spectral triple over the reduced crossed product. We return to these matters in \cref{sec:unbounded-KK}, and now turn to constructing the Dirac operator in each of the four combinations of parities.

\subsubsection{Odd times odd}


We first treat the case where both $p$ and $q$ are equal to $1$. 
The Hilbert space $G$ for the even unital spectral triple for $A \rtimes \Ga$ comes from the internal tensor product of Hilbert $C^*$-modules, so that
\[
G := \big( (\ell^2(\Ga,A)\ot \B C^n) \hot_A H \big)^{\op 2} .
\]
This Hilbert space is $\zz/2\zz$-graded with grading operator $\ga := \begin{pmatrix}1 & 0 \\ 0 & -1\end{pmatrix}$ so that the two factors in the above direct sum decomposition are, respectively, homogeneous even and odd. The reduced crossed product $A \rtimes \Ga$, which by definition is a unital $C^*$-subalgebra of $\B L\big( \ell^2(\Ga,A) \big)$, acts on the Hilbert space $G$ via the unital $*$-homomorphism
\[
z \mapsto \ma{cc}{{z \ot 1 \ot 1} & 0 \\ 0 & {z \ot 1 \ot 1}} .
\]
Since $A$ is a unital $C^*$-subalgebra of $\B L(H)$, we may identify the internal tensor product $(\ell^2(\Ga,A)\ot \B C^n) \hot_A H$ with the Hilbert space $\ell^2(\Ga,\B C^n \ot H)$, and we will often suppress this identification from the notation.
The abstract Dirac operator is defined as a \emph{tensor sum} (applying terminology suggested by Connes) of the two components $D_l$ and $D$. More precisely, for each $\la \in \B C$ we consider the unbounded operator
\[
\C D_l \ot 1 + \la \ot_\Na D \colon C_c\big(\Ga, \B C^n \ot \T{Dom}(D)\big) \to \ell^2\big(\Ga,\B C^n \ot H\big)
\]
which is given by the formula
{\begin{align}\label{eq:definition-of-Dirac-1}
(\C D_l \ot 1 + \la \ot_\Na D)\big( (\xi \ot \eta) \de_s  \big) := (l(s) \cd \xi \ot \eta) \de_s  + \la \cd (\xi \ot D\eta) \de_s 
\end{align}
for all $s \in \Ga$, $\xi \in \B C^n$ and $\eta \in \T{Dom}(D)$.}

The unbounded product operator $D_l \ti_\Na D$, also known as the tensor sum, is then defined as the closure of the unbounded symmetric operator
\[
\begin{pmatrix} 0 &  \C D_l \ot 1 + i \ot_\Na D \\ \C D_l \ot 1 - i \ot_\Na D & 0 \end{pmatrix}
\colon \big( C_c(\Ga, \T{Dom}(D)^{\op n})\big)^{\op 2} \to G .
\]

\begin{remark}\label{r:nabla}
{The subscript $\Na$ in the definition above refers to a Hermitian connection which is in fact used to lift the unbounded selfadjoint operator $D$ to the unbounded selfadjoint operator $1\otimes_{\Na} D$. More precisely, working directly with the internal tensor product $(\ell^2(\Ga,A) \ot \B C^n) \hot_A H$ one may define $1 \ot_{\Na} D$ as the sum
  \begin{equation}\label{eq:horilift}
  1 \ot_\Na D := c(\Na) + 1 \ot D \colon (C_c(\Ga,\C A) \ot \B C^n) \ot_{\C A} \T{Dom}(D)
  \to (\ell^2(\Ga,A) \ot \B C^n) \hot_A H ,
  \end{equation}
  where $c(\Na)$ is the operator which maps a simple tensor $a \de_s \ot \xi \ot \eta$ to the vector
  \[
c(\Na)(a \de_s \ot \xi \ot \eta) := \de_s \ot \xi \ot d(a) \eta .
\]
We shall give more details on this point of view in \cref{sec:unbounded-KK}. Notice in this respect that $1 \ot D$ is typically not well-defined on the module tensor product $(C_c(\Ga,\C A) \ot \B C^n) \ot_{\C A} \T{Dom}(D)$ since $D$ does not commute with all the elements in the coordinate algebra $\C A$. This defect is corrected by the extra term $c(\Na)$ involving the Hermitian connection $\Na$.} Notice that in the formula from \eqref{eq:definition-of-Dirac-1}, the extra term $c(\Na)$ does not appear since we are working with a tensor product over $\B C$ instead of the module tensor product over $\C A$. Let us in this connection clarify that the two Hilbert spaces $(\ell^2(\Ga,A) \ot \B C^n) \hot_A H$ and $\ell^2(\Ga,\B C^n \ot H)$ are unitarily isomorphic (mapping $(a \de_s \ot \xi) \ot \eta$ to $(\xi \ot a\eta) \de_s$) and using this identification, the two expressions for $1 \ot_\Na D$ from \eqref{eq:definition-of-Dirac-1} and \eqref{eq:horilift} coincide.
\end{remark}

\subsubsection{Even times even}

In the case where $p$ and $q$ are both even, the Hilbert space is the internal tensor product $G := (\ell^2(\Ga,A) \ot \B C^n) \hot_A H$ which is equipped with the action of the reduced crossed product given by $z \mapsto z \ot 1 \ot 1$. As in the odd times odd case, we identify $G$ with the Hilbert space $\ell^2(\Ga,\B C^n \ot H)$, but this time $\ell^2(\Ga,\B C^n \ot H)$ is $\zz/2\zz$-graded with grading operator given by the expression
\[
\ga_l \ot \ga \colon (\xi \ot \eta) \de_s \mapsto (\ga_l \cd \ga)(\xi \ot \eta) \de_s .
\]
The relevant unbounded product operator $D_l \ti_\Na D$ is then defined as the closure of the unbounded symmetric operator
\[
\C D_l \ot 1 + \ga_l \ot_\Na D \colon C_c\big(\Ga,\B C^n \ot \T{Dom}(D) \big) \to \ell^2(\Ga,\B C^n \ot H)
\]
given by the concrete formula
\[
(\C D_l \ot 1 + \ga_l \ot_\Na D)((\xi \ot \eta)\de_s) = (l(s) \cd \xi \ot \eta)\de_s + (\ga_l \cd \xi \ot  D(\eta) ) \de_s .
\]

\subsubsection{Odd times even and even times odd}
In the two cases where $p + q$ is odd, the Hilbert space is again $G := (\ell^2(\Ga,A) \ot \B C^n) \hot_A H$ equipped with the same action of the reduced crossed product as in the even times even case. As above, we identify $(\ell^2(\Ga,A) \ot \B C^n) \hot_A H$ with $\ell^2(\Ga, \B C^n \ot H)$ but we view this Hilbert space as being ungraded. For $p$ even and $q$ odd, the unbounded product operator is given by the same formula as in the case where $p$ and $q$ are both even. In the final case where $p$ is odd and $q$ is even, we define the unbounded product operator $D_l \ti_\Na D$ as the closure of the unbounded symmetric operator
\[
\C D_l \ot \ga + 1 \ot_\Na D \colon C_c\big(\Ga,\B C^n \ot \T{Dom}(D)\big) \to \ell^2(\Ga,\B C^n \ot H),
\]
which is given by the explicit formula
\[
(\C D_l \ot \ga + 1 \ot_\Na D)((\xi \ot \eta)\de_s) = (l(s) \cd \xi \ot \ga(\eta) ) \de_s + (\xi \ot D(\eta)) \de_s .
\]

The expected result is now as follows:

\begin{theorem}\label{thm:crossed-product-spectral-triple}
Let $(\A,H, D)$ be a {unital} spectral triple of parity $q$ and $\Gamma$ be a {countable} discrete group endowed with a matrix-valued length function $l$ of parity $p$. {Suppose that} $\Ga$ acts smoothly and metrically equicontinuously on the $C^*$-closure $A$ of $\A$, then $(\A\rtimes_{\alg}\Gamma, G, D_l\times_{\Na} D)$ is a unital spectral triple of parity $p+q$ (mod 2).
\end{theorem}

This theorem is well known and an outline of the proof in the ``odd times odd'' case can be found in \cite{HSWZ:STC} for integer valued length functions, but the same kind of reasoning may be applied for other parities and for matrix-valued length functions. We shall not dwell on these details, since the more subtle analysis of the $KK$-theoretic properties of {our data} in \cref{sec:unbounded-KK} will, as a side effect, also yield a proof of  \cref{thm:crossed-product-spectral-triple}.


\subsection{Construction of slip-norms}\label{subsec:slip-norms}
Let us stay in the setting {described in the statement of \cref{thm:crossed-product-spectral-triple}}. We are going to suppress the faithful representation of the reduced crossed product $A \rtimes \Ga$ on $\ell^2(\Ga,\B C^n \ot H)$ (but recall that this representation has been explained in \cref{ss:triples-on-crossed}). Regardless of the parities of the defining data we may consider the two essentially selfadjoint unbounded operators
\[
\begin{split}
& \C D_l \ot 1 \T{ and } 1 \ot_\Na D \colon C_c(\Ga,\B C^n \ot \T{Dom}(D)) \to \ell^2(\Ga,\B C^n \ot H) \T{ given by } \\
& (\C D_l \ot 1)((\xi \ot \eta) \de_s) := (l(s) \cd \xi \ot \eta) \de_s \quad \T{ and } \\
& (1 \ot_\Na D)((\xi \ot \eta) \de_s) := (\xi \ot D(\eta) ) \de_s .
\end{split}
\]
We denote the selfadjoint closures by $D_l \hot 1$ and $1 \hot_\Na D$, respectively, and these unbounded selfadjoint operators yield two derivations $d_V$ and $d_H$ on the algebraic crossed product $\C A \ralg \Ga$ as follows:\\

The \emph{vertical derivation} $d_V \colon \C A \ralg \Ga \to \B L\big(  \ell^2(\Ga,\B C^n \ot H)\big)$ is obtained by taking commutators with $D_l \hot 1$. This derivation is given explicitly by the formula
\begin{equation}\label{eq:verderi}
d_V( \pi_\al(x) \la_g) = M_{\varphi_g} \pi_\al(x) \la_g ,
\end{equation}
where $M_{\varphi_g}$ is the multiplication operator associated with the bounded map $\varphi_g \colon \Ga \to M_n(\B C)$ defined by $\varphi_g(s) := l(s) - l(g^{-1}s)$. Identifying the Hilbert space $\ell^2(\Gamma, \B C^n \ot H)$ with the internal tensor product $(\ell^2(\Gamma, A) \ot \B C^n)\hot_A H$, we see that $d_V$ agrees with $d_l\otimes 1$, where the derivation $d_l \colon \C A \ralg \Ga \to \B L( \ell^2(\Ga,A) \ot \B C^n)$ comes from the unbounded Kasparov module $\big( \C A \ralg \Ga, \ell^2(\Ga,A) \ot \B C^n, D_l \big)$ as described in \cref{ss:length}. In particular, we get that the slip-norm $L_l$ analysed in \cref{sec:slip-from-length} is given by the formula
\[
L_l(z) = \| d_V(z) \|_\infty \quad \T{for all } z \in \C A \ralg \Ga .
\]

The \emph{horizontal derivation} $d_H \colon \C A \ralg \Ga \to \B L\big( \ell^2(\Ga,\B C^n \ot H) \big)$ is obtained by taking commutators with $1 \hot_\Na D$. Recall that the action of $\Ga$ on $A$ is assumed to be {smooth and metrically equicontinuous}. 
Whenever $x \in \C A$ we therefore have a well-defined bounded operator $1 \ot d( \pi_\al(x))$ on the Hilbert space $\ell^2(\Ga,\B C^n \ot H)$ defined by the expression
\[
\big( 1 \ot d( \pi_\al(x)) \big)( (\xi \ot \eta) \de_s) := ( \xi \ot d( \al_s^{-1}(x))(\eta) ) \de_s .
\]
Indeed, the operator norm of this diagonal operator agrees with the supremum of the subset
\[
\big\{\| d( \al_s^{-1}(x)) \|_\infty \mid s \in \Ga \big\} \su [0,\infty),
\]
which is bounded because of metric equicontinuity. For $x \in \C A$ and $g \in \Ga$ it can be verified that $\pi_\al(x) \la_g$ preserves the core $C_c(\Ga,\B C^n \ot \T{Dom}(D))$ for $1 \hot_{\Na} D$. Let now $\xi \in \B C^n$, $\eta \in \T{Dom}(D)$ and $s \in \Ga$ be given and compute that
\[
\begin{split}
  [ 1 \hot_{\Na} D, \pi_\al(x) \la_g]\big((\xi \ot \eta)\de_s\big) &= \big( \xi \ot [ D, \al_{gs}^{-1}(x)] \eta \big) \de_{gs} \\
& = \big( 1 \ot d( \pi_\al(x)) \big) \la_g ( (\xi \ot \eta) \de_s) .
\end{split}
\]
This establishes the following explicit formula for the horizontal derivation
\begin{equation}\label{eq:d_H-formula}
d_H(  \pi_\al(x) \la_g)=\big( 1 \ot d( \pi_\al(x)) \big) \la_g.
\end{equation}
Notice that the horizontal derivation on the algebraic crossed product $\C A \ralg \Ga$ gives rise to a slip-norm defined by $z \mapsto \| d_H(z) \|_\infty$.

Finally, we let $L_{D_l \ti_\Na D}$  denote the slip-norm on $\C A \ralg \Ga$ which arises from the unital spectral triple $\big( \C A \ralg \Ga, G, D_l \ti_\Na D\big)$ by taking norms of commutators. The following convenient result, which can be achieved by going through the argument presented as part of \cite[Theorem 7.1]{Kaad2023}, shows how the three slip-norms are related.

\begin{lemma}\label{l:equislip}
  In all the four combinations of parities, we have the inequalities
  \[
\max\big\{\| d_V(z) \|_\infty, \| d_H(z) \|_\infty \big\} \leq L_{D_l \ti_\Na D}(z) \leq 2 \cd \max\big\{\| d_V(z) \|_\infty, \| d_H(z) \|_\infty \big\}
\]
for all $z \in \C A \ralg \Ga$.
\end{lemma}
\begin{proof}
Denote by $d$ the derivation arising from $D_l \ti_\Na D$; i.e.~$d(z)=\T{cl}([D_l \ti_\Na D, z])$ for $z\in \C \A \rtimes_{\alg}\Gamma$.
The proof follows the same overall strategy in all four cases, {so we just give the details in the cases where $p$ and $q$ are both even or $p$ and $q$ are both odd.}

If $p$ and $q$ are both even, we may promote $\gamma_l$ to a selfadjoint unitary operator on $\ell^2(\Gamma, \B C^n \ot H)$ by setting $\gamma_l((\xi \ot \eta)\delta_s)=(\gamma_l\cdot \xi \ot \eta)\delta_s$  and a direct computation shows that 
$d(z)=d_V(x) + \gamma_l d_H(z)$ from which the upper bound in the lemma follows. Using that $\gamma_l$ commutes with $d_H(z)$ and anti-commutes with $d_V(z)$, {we obtain the lower bound by considering} the two bounded operators $d(z) \pm \gamma_l d(z)\gamma_l$. 
%

If both $p$ and $q$ are odd, one has the formula
\[
d(z)=d_V(z)\sigma_1 +d_H(z)\sigma_2 \quad \text{ with } \sigma_1=\begin{pmatrix} 0 & 1 \\ 1 & 0 \end{pmatrix} \ \text{ and } \ \sigma_2=\begin{pmatrix} 0 & i \\ -i &0 \end{pmatrix},
\]
from which the upper bound follows. Using that $\sigma_1\sigma_2=-\sigma_2\sigma_1$, {we obtain the relevant lower bound by considering} the bounded operators $d(z) +\sigma_1d(z)\sigma_1$ and $d(z) +\sigma_2d(z)\sigma_2$. 
\end{proof}

Our next aim is to utilise the equivalence of seminorms from \cref{l:equislip} together with \cref{thm:main-theorem-crossed} to obtain conditions ensuring that the unital spectral triple $(\A\rtimes_{\alg}\Gamma, G, D_l\times_{\Na} D)$ is a spectral metric space.

\subsection{Proof of \cref{introthm:spectral-metric-space}}\label{sec:spectral-metrics-on-crossed}
{Throughout this section we consider a fixed unital spectral triple $(\C A,H,D)$ and a fixed countable discrete group $\Ga$ which acts by $*$-automorphisms on the unital $C^*$-algebra $A$ obtained by closing $\C A \su \B L(H)$ in the operator norm. We are moreover assuming that $\Ga$ is equipped with a proper matrix length function $l \colon \Ga \to M_n(\B C)$.}


{Let us consider the supremum norm $\nvert{\cd}_\infty$ on $C_c(\Ga,\B C)$ defined by the formula
\[
\nvert{f}_\infty := \sup_{s \in \Ga} | f(s) |
\]
and record that the supremum norm is order-preserving. Combining the supremum norm with the slip-norm $L_D \colon \C A \to [0,\infty)$ coming from the unital spectral triple $(\C A,H,D)$ we obtain the seminorm
\[
L_H^\infty \colon \C A \ralg \Ga \to [0,\infty) \quad L_H^\infty(z) := \nvert{L_D \ci z }_\infty . 
  \]
  Recall in this respect that $L_D \ci z \colon \Ga \to [0,\infty)$ refers to the compactly supported function given by $(L_D \ci z)(g) = L_D( E_g(z))$ for all $g \in \Ga$.

    We moreover apply the notation $L_\infty \colon \C A \ralg \Ga \to [0,\infty)$ for the slip-norm defined by
      \[
L_\infty(z) := \max\{L_l(z) , L_H^\infty(z), L_H^\infty(z^*) \} .
\]

In the case where the action of $\Ga$ is smooth and metrically equicontinuous (with respect to $(\C A,H,D)$) we recall from \cref{thm:crossed-product-spectral-triple} that $\big( \C A \ralg \Ga, G, D_l \ti_\Na D \big)$ is a unital spectral triple. The corresponding slip-norm is denoted by $L_{D_l \ti_\Na D}$.
    
  \begin{lemma}\label{l:specdomi}
Suppose that the action of $\Ga$ on $A$ is smooth and metrically equicontinuous. We have the inequalities
\[
L_H^\infty(z) \leq \| d_H(z) \|_\infty \ \mbox{ and }  \ L_\infty(z) \leq L_{D_l \ti_\Na D}(z) \quad \T{for all } z \in \C A \ralg \Ga .
  \]
  \end{lemma}
  \begin{proof}
    The second inequality follows immediately from the first one by an application of \cref{l:equislip} and the observation that $L_l(z) = \| d_V(z) \|_\infty$ for all $z \in \C A \ralg \Ga$ (see the discussion in \cref{subsec:slip-norms}). We therefore focus on proving the first inequality.

Consider an element $z = \sum_{s \in \Ga} \pi_\al(x_s) \la_s \in \C A \ralg \Ga$ and let $g \in \Ga$. For every pair of vectors $\eta$ and $\ze \in H$, it is then a consequence of \eqref{eq:d_H-formula} that we have the explicit formulae
   \[
 \begin{split}
   & \binn{(e_1 \ot \ze) \de_e, d_H(z) (e_1 \ot \eta) \de_{g^{-1}} } 
   = \sum_{s \in \Ga} \binn{(e_1 \ot \ze) \de_e , \big( 1 \ot d(\pi_\al(x_s))\big) (e_1 \ot \eta) \de_{sg^{-1}} } \\
   & \quad = \binn{(e_1 \ot \ze) \de_e , \big(1 \ot d(\pi_\al(x_g)) \big) (e_1 \ot \eta) \de_e }
    = \inn{\ze, d( x_g) \eta } .
\end{split}
    \]
It follows from the above computation that we have the inequality
\[
\| d_H(z) \|_\infty \geq \| d(x_g) \|_\infty = L_D(x_g) = (L_D \ci z)(g)  . 
\]
This in turn shows that $L_H^\infty(z) \leq \| d_H(z) \|_\infty$.
\end{proof}

\begin{proof}[Proof of \cref{introthm:spectral-metric-space} ]
  Recall that the action of $\Ga$ on $A$ is smooth and metrically equicontinuous by assumption. It is furthermore assumed that $(\C A,H,D)$ and $(\B C \Ga, \ell^2(\Ga) \ot \B C^n , D_l)$ are spectral metric spaces and that $\Ga$ is amenable.

  We need to verify that the unital spectral triple $( \C A \ralg \Ga, G, D_l \ti_\Na D)$ is a spectral metric space. However, since the supremum norm $\nvert{\cd}_\infty$ is order-preserving, we get from \cref{thm:main-theorem-crossed} that the pair $(\C A \ralg \Ga, L_\infty)$ is a compact quantum metric space. It now follows from \cref{l:specdomi} and \cref{thm:Rieffels-criterion} that the pair $(\C A \ralg \Ga, L_{D_l \ti_\Na D})$ is a compact quantum metric space.  
\end{proof}}

\section{Unbounded $KK$-theory and proof of \cref{t:unbprod}}\label{sec:unbounded-KK}
{In this section, we work in the setting where $(\C A,H,D)$ is a unital spectral triple on the unital $C^*$-algebra $A$ and $\Ga$ is a countable discrete group which acts on $A$ in a smooth and metrically equicontinuous way. We shall moreover fix a proper matrix length function $l \colon \Ga \to M_n(\B C)$. We are thus in the context of \cref{thm:crossed-product-spectral-triple} so that $\big( \C A \ralg \Ga, G, D_l \ti_\Na D\big)$ is a unital spectral triple on the reduced crossed product $A \rtimes \Ga$. Our aim is now to interpret this unital spectral triple as the unbounded Kasparov product of the unbounded Kasparov module $\big( \C A \ralg \Ga, \ell^2(\Ga,A) \ot \B C^n, D_l \big)$ constructed in \cref{ss:length} and the unital spectral triple $(\C A,H,D)$. Let us emphasize that this point of view is different from the point of view mentioned in \cite[Remark 2.9]{HSWZ:STC} involving an equivariant version of the \emph{external} unbounded Kasparov product. Indeed, we are in this section working with the \emph{internal} unbounded Kasparov product which was introduced in \cite{Mes:UKC} and further developped in \cite{KaLe:SFU,MeRe:NMU}. 

Before we continue any further let us recall that the unbounded Kasparov module $\big( \C A \ralg \Ga, \ell^2(\Ga,A) \ot \B C^n, D_l \big)$ gives rise to a class in $KK$-theory $[D_l] \in KK^p(A \rtimes \Ga,A)$ (where $p = 0$ if $l$ is graded and $p = 1$ otherwise). This class in $KK$-theory is referred to as the \emph{Baaj-Julg bounded transform} of the unbounded Kasparov module and it is represented by the Kasparov module $\big( \ell^2(\Ga,A) \ot \B C^n, D_l(1 + D_l^2)^{-1/2} \big)$; see \cite{BaJu:TBK}. Similarly, we get a class in $K$-homology $[D] \in K^q(A)$ (where $q = 0$ in the graded case and $q = 1$ otherwise) by applying the Baaj-Julg bounded transform to the unital spectral triple $(\C A,H,D)$. Thus, in the case where the unital $C^*$-algebra $A$ is separable we may form the internal Kasparov product of our two classes, obtaining a $K$-homology class $[D_l] \hot_A [D] \in K^{p + q}( A \rtimes \Ga)$ for the reduced crossed product (the parity $p+q$ should be understood modulo $2$); see \cite{Kas:OKF}.

The point of applying the unbounded Kasparov product to our data is that this technique gives a systematic way of representing the $K$-homology class $[D_l] \hot_A [D]$ by a unital spectral triple on the reduced crossed product. Thus, letting $[D_l \ti_\Na D] \in K^{p+q}(A \rtimes \Ga)$ denote the $K$-homology class obtained as the Baaj-Julg bounded transform of the unital spectral triple $\big( \C A \ralg \Ga, G, D_l \ti_\Na D\big)$, we get the identity
\[
[D_l \ti_\Na D] = [D_l] \hot_A [D] .
\]
Ultimately, we are here applying \cite{Kuc:KKU} which provides conditions ensuring that a given unbounded Kasparov module represents an internal Kasparov product. As a byproduct, it is also a consequence of the unbounded $KK$-theoretic point of view that the data $\big( \C A \ralg \Ga, G, D_l \ti_\Na D\big)$ is in fact a unital spectral triple, but this also follows from \cite{HSWZ:STC} (at least for $q = 1$). 

We are in this section focusing on the case where $p$ and $q$ are both odd and our approach to the unbounded Kasparov product follows the scheme presented in \cite{KaLe:SFU}. The remaining combinations of parities can be dealt with by using similar techniques and the interested reader can consult \cite{Mes:UKC,MeRe:NMU} for a detailed treatment in the setting where $p$ and $q$ are both even. Notice that \cite{Mes:UKC,MeRe:NMU} are written for $\zz/2\zz$-graded $C^*$-algebras and the formulas presented there can thus be adapted to take care of the remaining combinations of parities as well (by performing suitable manipulations with Clifford algebras). We shall moreover simplify matters slightly by assuming that $n = 1$ so that our ungraded unbounded Kasparov module is of the form $\big( \C A \ralg \Ga, \ell^2(\Ga,A), D_l \big)$.}


In order to prove \cref{t:unbprod} using unbounded $KK$-theoretic techniques it is necessary to write down a \emph{correspondence} from $(\C A \ralg \Ga, \ell^2(\Ga,A), D_l)$ to $(\C A,H,D)$, see \cite[Theorem 6.7]{KaLe:SFU}. According to \cite[Definition 6.3]{KaLe:SFU} the first step is to upgrade $\C A$ to an \emph{operator $*$-algebra} and to find an appropriate \emph{Hermitian operator module} sitting inside $E := \ell^2(\Ga,A)$. Both of these structures do in fact come from the closable derivation $d \colon \C A \to \B L(H)$.\\

We let $A_d \su A$ denote the domain of the closure of $d$ and we view $A_d$ as a closed unital subalgebra of $\B L(H \op H)$ via the injective unital algebra homomorphism
\[
\pi_d \colon A_d \to \B L(H \op H) \quad \pi_d(x) := \ma{cc}{x & 0 \\ d(x) & x} .
\]
In this fashion $A_d$ becomes a unital operator algebra; {see e.g. \cite[Chapter 2]{BlMe:OAM} for more details.} It turns out that $A_d$ carries even more structure since the involution from $A$ induces a completely isometric involution on $A_d$ and this entails that $A_d$ is a unital operator $*$-algebra ($A_d$ is not a $C^*$-algebra since $\pi_d$ does not preserve the adjoint operation). By construction it holds that the inclusion $A_d \to A$ is a completely contractive unital $*$-homomorphism and similarly that the derivation $d : A_d \to \B L(H)$ is completely contractive.\\

Let us now construct the relevant Hermitian operator module $E_\Na \su E$. We start out by observing that the elements in the Hilbert $C^*$-module $\ell^2(\Ga,\B L(H))$ can be viewed as bounded operators from $H$ to $\ell^2(\Ga,H)$. Indeed, an element $T = \sum_{s \in \Ga} T_s \de_s$ which belongs to $\ell^2(\Ga,\B L(H))$ gives rise to the bounded operator given by 
\begin{equation}\label{eq:boundedrep}
T(\ze) := \sum_{s \in \Ga} T_s(\ze) \de_s \in \ell^2(\Ga,H) \quad \T{for all } \ze \in H .
\end{equation}
{In this fashion, we have a well-defined adjoint $T^* \in \B L\big( \ell^2(\Ga,H), H \big)$ given by the explicit formula
\begin{equation}\label{eq:adjoint}
T^*\Big( \sum_{s \in \Ga}\ze_s \de_s \Big) = \sum_{s \in \Ga} T_s^*(\ze_s) .
\end{equation}
Using this description,} the inner product of two elements $T$ and $S$ in the Hilbert $C^*$-module $\ell^2(\Ga,\B L(H))$ is just given by the product of bounded operators $\inn{T,S} = T^* S$. Since the unital $C^*$-algebra $A$ is identified with a unital $C^*$-subalgebra of $\B L(H)$ we may also view the elements in the Hilbert $C^*$-module $E = \ell^2(\Ga,A)$ as elements inside $\ell^2(\Ga,\B L(H))$ and the above considerations therefore apply here as well.

The next step is to introduce the closable Hermitian $d$-connection
\[
\Na \colon C_c(\Ga, \C A) \to \ell^2(\Ga,\B L(H)), \quad \Na( a \de_s) := d(a) \de_s  .
\]
The fact that $\Na$ is a Hermitian $d$-connection means that $\Na$ is $\B C$-linear and satisfies the identities
\begin{equation}\label{eq:leibherm}
\Na( \xi \cd a) = \Na(\xi) \cd a + \xi \cd d(a) \, \, \T{ and } \, \, \, 
d( \inn{\xi,\eta} ) = \xi^* \cd \Na(\eta) - \Na(\xi)^* \cd \eta
\end{equation}
for all $\xi,\eta \in C_c(\Ga,\C A)$ and all $a \in \C A$. Notice that the adjoint operations in the second equation of \eqref{eq:leibherm} refer to adjoints of bounded operators from $H$ to $\ell^2(\Ga,H)$ using the representation from \eqref{eq:boundedrep} {and the associated formula from \eqref{eq:adjoint}.}

We define $E_\Na \su E$ as the domain of the closure of $\Na$. In order to provide $E_\Na$ with the structure of a Hermitian operator module we first identify $E_\Na$ with a closed subspace of the bounded operators $\B L\big( H^{\op 2}, \ell^2(\Ga,H)^{\op 2}  \big)$. This is carried out via the linear map
\[
\pi_\Na \colon E_\Na \to \B L\big( H^{\op 2}, \ell^2(\Ga,H)^{\op 2}  \big), \quad \pi_\Na(\xi) := \ma{cc}{\xi & 0 \\ \Na(\xi) & \xi},
\]
where we recall that both $\xi$ and $\Na(\xi)$ can be viewed as bounded operators from $H$ to $\ell^2(\Ga,H)$ (in the case where $\xi$ belongs to $E_\Na$). It then follows from the first identity in \eqref{eq:leibherm} that $E_\Na$ becomes a right operator module over the {operator algebra} $A_d$ with right module structure induced by the right module structure {on $E$; see e.g. \cite[Chapter 3]{BlMe:OAM} for more details.} Notice that the relevant completely contractive nature of the right module structure follows since the unital algebra homomorphism $\pi_d$ and the linear map $\pi_\Na$ are compatible in so far that
\[
\pi_\Na(\xi \cd a) = \pi_\Na(\xi) \cd \pi_d(a) \quad \T{for all } \xi \in E_\Na \T{ and } a \in A_d .
\]
Because of the second identity in \eqref{eq:leibherm} we get that the $A$-valued inner product on $E$ induces a completely contractive hermitian form $\inn{\cd,\cd} \colon E_\Na \ti E_\Na \to A_d$. {Notice in this respect that
\[
U \pi_\Na(\xi)^* U \cd \pi_\Na(\eta) = \pi_d( \inn{\xi,\eta})  \quad \T{for all } \xi, \eta \in E_\Na,
\]
where $U$ is the selfadjoint unitary operator given by $U = \ma{cc}{0 & i \\ -i & 0}$.} We therefore get that $E_\Na$ is a Hermitian operator module over $A_d$ and it follows by construction that both the inclusion $E_\Na \to E$ and the connection $\Na \colon E_\Na \to \ell^2(\Ga,\B L(H))$ are completely contractive operations. For more details on these constructions we refer the reader to \cite{BlKaMe:OAG}.\\

Under the conditions stated in \cref{t:unbprod}, we are going to show that the pair $(E_\Na,\Na)$ is a correspondence from $(\C A \ralg \Ga, \ell^2(\Ga,A), D_l)$ to $(\C A,H,D)$. {One of the main conditions on a correspondence (see \cite[Definition 6.3]{KaLe:SFU}) concerns the structure of the associated Hermitian operator module. In the terminology of \cite{KaLe:SFU} it is required that this Hermitian operator module is an \emph{operator $*$-module}, which means that the Hermitian operator module sits as a direct summand in a standard Hermitian operator module, see \cite[Definition 3.4]{KaLe:SFU}. We therefore spend some more time analysing the structure of the Hermitian operator module $E_\Na$ focusing on the case where the countable discrete group $\Ga$ is infinite (the finite case is of course much easier to deal with).}

Since $\Ga$ is countably infinite we may identify the Hilbert $C^*$-module $E = \ell^2(\Ga,A)$ with the standard module $\ell^2(\nn,A)$ via an isomorphism of sets $\Ga \cong \nn$. Our isomorphism of Hilbert $C^*$-modules then induces a completely isometric isomorphism of Hermitian operator modules between $E_\Na \su E$ and the standard column module $\ell^2(\nn,A_d) \su \ell^2(\nn,A)$ (which is denoted by $A_d^\infty$ in \cite[Definition 3.3]{KaLe:SFU} and by $C_\nn(A_d)$ in \cite[Section 3.1.13]{BlMe:OAM}). Notice that the Hermitian operator module structure of the standard column module is induced by the Hilbert $C^*$-module structure of $\ell^2(\nn,A)$. {In particular, we get that $E_\Na$ is an operator $*$-module.}  

In the present situation, the unbounded selfadjoint operator $1 \hot_{\Na} D$ constructed in \cite[Section 5.1]{KaLe:SFU} agrees with the closure of the diagonal operator
\[
\T{diag}(D) \colon C_c(\Ga, \T{Dom}(D)) \to \ell^2(\Ga,H), \quad \T{diag}(D)(\xi \de_s) = D(\xi) \de_s .
\]
The remaining leg of the unbounded product operator is the unbounded selfadjoint operator $D_l \hot 1$ which is obtained as the closure of the unbounded symmetric operator
\[
\C D_l \ot 1 \colon C_c(\Ga,\T{Dom}(D)) \to \ell^2(\Ga,H), \quad (\C D_l \ot 1)(\xi \de_s) = l(s) \cd \xi \de_s . 
\]

\begin{proposition}\label{p:correspond}
Suppose that the action of $\Ga$ on $A$ is {smooth and metrically} equicontinuous with respect to the unital spectral triple $(\C A,H,D)$. Suppose moreover that $l \colon \Ga \to \B R$ is a proper length function. Then the pair $(E_\Na,\Na)$ is a correspondence from $(\C A \ralg \Ga, \ell^2(\Ga,A), D_l)$ to $(\C A,H,D)$.
\end{proposition}
\begin{proof}
  We only need to verify condition $(3)$ and $(4)$ of \cite[Definition 6.3]{KaLe:SFU}.
 
  Let $x \in \C A$ and $g \in \Ga$ be given. To verify condition $(3)$ it suffices to show that $\pi_\al(x) \la_g$ preserves the core $C_c(\Ga,\T{Dom}(D))$ for $1 \hot_{\Na} D$ and that the commutator
  \[
[ 1 \hot_{\Na} D, \pi_\al(x) \la_g] \colon C_c(\Ga,\T{Dom}(D)) \to \ell^2(\Ga,H)
\]
extends to a bounded operator on $\ell^2(\Ga,H)$. But this was already established in \cref{subsec:slip-norms} {(see the discussion regarding the horizontal derivation $d_H$), so we are left with verifying condition $(4)$.}

{This remaining condition says} that the ``commutator'' between $D_l \hot 1$ and $1 \hot_\Na D$ is relatively bounded in the sense which is carefully spelled out and applied in \cite{KaLe:LGP}. In our case, this condition is particularly easy to verify by considering the core $C_c(\Ga,\T{Dom}(D))$ for $1 \hot_\Na D$. {Indeed, it clearly holds that $(D_l \hot 1 - i \mu)^{-1}$ preserves the core $C_c(\Ga,\T{Dom}(D))$ for all $\mu \in \B R \sem \{0\}$. Moreover, for every $\xi \in \T{Dom}(D)$ and $s \in \Ga$ we have that
  \[
[ D_l \hot 1, 1 \hot_\Na D] ( \xi \de_s) = l(s) \cd D(\xi) \de_s - D( l(s) \cd \xi) \de_s = 0,
  \]
entailing that the relevant commutator of unbounded operators is equal to zero.} This establishes condition $(4)$ and hence that $(E_\Na, \Na)$ is a correspondence.
\end{proof}

With \cref{p:correspond} at hand, we are now ready to prove the main result of this subsection using unbounded $KK$-theoretic techniques. Strictly speaking we are only giving the proof for $n = 1$ and $\Ga$ being countably infinite, but the same methods apply with minor modifications in the more general setting where $n \in \N$ and $\Ga$ is allowed to be finite. As explained in the beginning of this subsection we have chosen to focus on the case where both of the parities $p$ and $q$ are odd. 


\begin{proof}[Proof of \cref{t:unbprod} (odd times odd)]
  By combining \cite[Theorem 6.7]{KaLe:SFU} and \cite[Theorem 7.5]{KaLe:SFU} with \cref{p:correspond} we only need to verify that the unbounded product operator $D_l \ti_\Na D$ which we constructed {right before \cref{r:nabla}} does indeed agree with the unbounded selfadjoint operator
  \[
     \begin{pmatrix}
  0 & D_l \hot 1 + i (1 \hot_\Na D) \\
  D_l \hot 1 - i (1 \hot_\Na D) & 0  \end{pmatrix}
 \hspace{-0.07cm} \colon \hspace{-0.07cm} \big( \T{Dom}(D_l \hot 1) \cap \T{Dom}(1 \hot_\Na D) \big)^{\op 2}
  \to \ell^2(\Ga,H)^{\op 2} .
  \]
  
  The description of the previous operator as selfadjoint on the intersection of domains $\big( \T{Dom}(D_l \hot 1) \cap \T{Dom}(1 \hot_\Na D) \big)^{\op 2}$ is contained in the statement of \cite[Theorem 6.7]{KaLe:SFU}. However, this part regarding selfadjointness is in fact a consequence of \cite[Theorem 7.10]{KaLe:LGP}.
  
  Since the two unbounded operators in question already agree on the core $C_c(\Ga, \T{Dom}(D))^{\op 2}$ for the unbounded product operator $D_l \ti_\Na D$, it suffices to establish that $D_l \ti_\Na D$ is indeed selfadjoint. {This does in turn hold,} if we can show that the restrictions of the two unbounded operators $D_l \ti_\Na D \pm i$ to the core $C_c(\Ga, \T{Dom}(D))^{\op 2}$ both have dense image in $\ell^2(\Ga,H)^{\op 2}$. But this follows by recording that the unbounded operator
  \[
T_s:=\begin{pmatrix}0 & l(s) + i D \\ l(s) - i D & 0\end{pmatrix} \colon \T{Dom}(D)^{\op 2} \to H^{\op 2}
\]
is selfadjoint for all $s \in \Ga$. {Indeed, this implies} that both of the unbounded operators $T_s\pm i$ are surjective and hence that the dense subspace $C_c(\Gamma, H)^{\op 2} \su \ell^2(\Ga,H)^{\op 2}$ is contained in the range of the two unbounded operators $D_l \ti_\Na D \pm i$.
\end{proof}

\bibliographystyle{siam}
\bibliography{bibliography.bib}
	
\end{document}